\documentclass[12pt,reqno]{amsart}

\usepackage{graphics}
\usepackage{latexsym}
\usepackage[dvips]{epsfig}
\usepackage{color}
\usepackage{graphicx}
\usepackage{amsmath}
\usepackage{amssymb}
\usepackage{amsfonts}
\usepackage{eufrak}
\usepackage{amstext}
\usepackage{amsopn}
\usepackage{amsbsy}
\usepackage{amscd}
\usepackage{amsxtra}
\usepackage{amsthm}
\usepackage{bm}
\usepackage{CJK}

\newcommand{\equ}[1]{(\ref{#1})}

\newcommand{\R}{\mathbb{R}}

\def\theequation{\thesection.\@arabic\c@equation}
\makeatother

\renewcommand{\theequation}{\thesection.\arabic{equation}}
\newtheorem{lemma}{Lemma}[section]

\newtheorem{proposition}{Proposition}[section]
\newtheorem{corollary}{Corollary}[section]
\newtheorem{remark}{Remark}[section]
\newtheorem{theorem}{Theorem}[section]

\newcommand{\ve} {\varepsilon}
\newcommand{\be}{\begin{equation}}
\newcommand{\ben}{\begin{equation*}}
\newcommand{\ee}{\end{equation}}
\newcommand{\een}{\end{equation*}}
\newcommand{\BL}{\begin{lemma}}
\newcommand{\EL}{\end{lemma}}
\newcommand{\BT}{\begin{theorem}}
\newcommand{\ET}{\end{theorem}}
\newcommand{\BP}{\begin{proposition}}
\newcommand{\EP}{\end{proposition}}
\newcommand{\BC}{\begin{corollary}}
\newcommand{\EC}{\end{corollary}}
\def\bs{\begin{split}}
\def\es{\end{split}}

\begin{document}
\title[Interior spike solutions]{An optimal bound on the number of interior spike solutions for the Lin-Ni-Takagi problem }

\author{Weiwei Ao}
\address{Department of Mathematics, Chinese University of Hong Kong, Shatin, Hong Kong.  {\sl wwao@math.cuhk.edu.hk}}
\author{ Juncheng Wei}
\address{Department of Mathematics, Chinese University of Hong Kong, Shatin, Hong Kong,and Department of Mathematics, University of British Columbia, Vancouver, B.C., Canada, V6T 1Z2.
{\sl wei@math.cuhk.edu.hk}}

\author{Jing Zeng}
\address{Department of Mathematics, Fujian Normal University, Fuzhou, Fujian, China}

\begin{abstract}
We consider the following singularly perturbed Neumann problem
\begin{eqnarray*}
\ve^2 \Delta u -u +u^p = 0 \quad {\mbox {in}} \quad \Omega, \quad u>0 \quad {\mbox {in}} \quad \Omega, \quad
{\partial u \over \partial \nu}=0 \quad {\mbox {on}} \quad \partial \Omega,
\end{eqnarray*}
where $p$ is subcritical and $\Omega$ is a smooth and bounded domain in $\R^n$ with its unit outward normal $\nu$. Lin-Ni-Wei \cite{LNW} proved  that there exists $\ve_0$ such that for $0<\ve<\ve_0$ and for each integer $k$ bounded by
\begin{equation}
1\leq k\leq \frac{\delta(\Omega,n,p)}{ (\ve |\log \ve |)^n}
\end{equation}
where $\delta(\Omega,n,p)$ is a constant depending only on $\Omega$, $p$ and $n$, there exists a solution with $k$ interior spikes. We show that the bound on $k$ can be improved to
\begin{equation}
1\leq k\leq \frac{\delta(\Omega,n,p)}{ \ve^n},
\end{equation}
which is optimal.
\end{abstract}



\date{}\maketitle

\setcounter{equation}{0}
\section{Introduction and statement of main results}
Of concern is the following Lin-Ni-Takagi problem (\cite{LNT})
\begin{equation}\label{p}
\left\{\begin{array}{c}
\ve^2\Delta u-u+u^p=0 \mbox{ in } \Omega\\
u>0 \mbox{ in } \Omega\\
\frac{\partial u}{\partial \nu}=0  \mbox{ on }\partial \Omega,
\end{array}
\right.
\end{equation}
where $p$ satisfies $1<p<+\infty$ for $n=2$ and $1<p<\frac{n+2}{n-2}$ for $n\geq 3$ and $\Omega$ is bounded, smooth domain in $\R^n$ with its unit outward normal $\nu$.

\medskip

Problem (\ref{p}) arises in many applied models concerning biological pattern formations. For instance, it gives rise to steady states in the Keller-Segel model of the chemotactic aggregation of the cellular slime molds and it also plays an important role in the Gierer-Meinhardt model describing the regeneration phenomena of hydra. See \cite{GM}, \cite{KS} and \cite{LNT} for more details.

\medskip

Problem (\ref{p}) has been studied extensively for the last twenty years. In the pioneering  paper \cite{LNT}, Lin, Ni and Takagi proved the a priori estimates and existence of least energy solutions to (\ref{p}), that is,  a solution $u_\epsilon$ with minimal energy. Furthermore, Ni and Takagi showed in  \cite{NT,NT1} that, {\it for each $\epsilon >0$ sufficiently small, $u_\epsilon$ has a spike at the most curved part of the boundary, i.e., the region where the mean curvature attains maximum value}.

\medskip

Since the publication of \cite{NT1}, problem (\ref{p}) has received  a great deal of attention and significant progress has been made. More specifically, solutions with multiple boundary peaks  as well as  multiple interior peaks have been established. (See \cite{DFW}-\cite{DFW1}, \cite{GPW}-\cite{GWW}, \cite{L}-\cite{LNW}, \cite{NW}-\cite{WW1} and the references therein.) In particular, it was established in Gui and Wei \cite{GW1} that {\it for any two given integers $k\geq 0, l\geq 0$ and $ k+l >0$, problem (\ref{p}) has a solution with exactly $k$ interior spikes and $l$ boundary spikes} for every $\epsilon$ sufficiently small. Furthermore, Lin, Ni and Wei \cite{LNW} showed that there are at least $ \frac{\delta(n,p,\Omega)}{ (\epsilon |\log \epsilon|)^n}$ number of interior spikes.
On the other hand, problem (\ref{p}) also admits higher dimensional concentrations. (See  \cite{N1}.)   For results in this direction, we refer to \cite{AMN}, \cite{M1}-\cite{MM1}. In particular, we mention the results of Malchiodi and Montenegro \cite{MM,MM1} on the existence of solutions concentrating on the {\em whole boundary} provided that the sequence $\ve$ satisfies some gap condition.

\medskip

In this paper, we shall address the question of the maximal possible number of spikes, in terms of small parameter $\ve>0$, that a solution of (\ref{p}) could have. Note that since $p$ is subcritical, the solutions to (\ref{p}) is uniformly bounded (Lin-Ni-Takagi \cite{LNT}). Thus the energy bound for solutions of (\ref{p}) is $O(1)$. On the other  hand, each spike contributes to at least $O(\ve^n)$ energy. This implies that the number of interior spikes can not exceed $ O(\ve^{-n})$.  Our main result, Theorem \ref{theo1} below, asserts that for every positive integer $k\leq \frac{\delta_{\Omega,n,p}}{\ve^n}$, where $\delta (\Omega,n,p)$ is a  constant depending only on $n,p$ and $\Omega$, problem (\ref{p}) has a solution with exactly $k$ peaks. This gives an optimal bound on the number of interior spikes.

\medskip

Our proof uses a {\em ``localized energy method''}  as in \cite{GW} and \cite{LNW}. There are two main difficulties. First, the distance between spikes is assumed only to be $O(\ve)$. In the Liapunov-Schmidt reduction process, we have to prove that all the estimates are uniform with respect to the integer $k$. Second, we have to detect the difference in the energy when spikes move to the boundary of the configuration space. A crucial estimate is Lemma 5.1, in which we prove that the accumulated error can be controlled from step $k$ to  step $ k+1$. To prove Lemma 5.1, we have to perform a secondary Liapunov-Schmidt reduction. This seems to be new.

We now state the main result in this paper.

\begin{theorem}\label{theo1}
There exists an $\ve_0>0$ such that for $0<\ve<\ve_0$, and any positive integer $k$ satisfying
\begin{equation}
1\leq k\leq \frac{\delta (\Omega,n,p) }{\ve^n},
\end{equation}
where $\delta (\Omega,n,p)$ is a constant depending on $n,\Omega$ and $p$ only, problem (\ref{p}) has a solution $u_\ve$ that possesses exactly $k$ local maximum points.
\end{theorem}

\noindent
\begin{remark}
 {\em As mentioned earlier,  the upper bound for $k$ is the best possible. As far as we know, the only result on the optimal upper bound for the number of spikes  is the one-dimensional situation. In a series of papers \cite{FMT2}-\cite{FMT3}, Felmer-Martinez-Tanaka studied the following singularly perturbed nonlinear Schr\"odinger equation
\begin{equation}
\label{1d}
\ve^2 \Delta u -V(x) u+ u^p=0, u >0, u \in H^1 (\R).
\end{equation}
They constructed solutions to (\ref{1d}) with $\frac{C}{\ve}$ number of spikes. Extension to Gierer-Meinhardt system can be found in \cite{FMT1}.
 Related construction can also be found in del Pino-Felmer-Tanaka \cite{DFT}.
}

\end{remark}

\noindent
 \begin{remark}
 {\em  An interesting problem is to study the {\em homogenization} of the  measure $ \ve^{-n} |\nabla u|^2 dx$.  We expect that it will approach some kind of Lebesgue measure.   As $\ve \to 0$, the locations of the maximum points should approach to some sphere-packing positions. }
\end{remark}

\noindent
\begin{remark}
{\em It is clear that the proofs of Theorem \ref{theo1} can be applied to a large class of singularly perturbed problems
\begin{equation}
\left\{\begin{array}{l}
\epsilon^2 \Delta u- u+ f(u)=0 \ \mbox{in} \ \Omega \\
 u>0 \ \mbox{in} \ \Omega, \ \frac{\partial u}{\partial \nu} =0 \ \mbox{on} \ \partial \Omega,
\end{array}
\right.
\end{equation}
where $ f(u)$ satisfies the conditions (f1)-(f3) stated in \cite{LNW}.
}
\end{remark}

\medskip

The paper is organized as follows. Notations, preliminaries and some useful estimates are explained in Section 2. Section 3 contains the study of a linear problem that is the first step in the Lyapunov-Schmidt reduction process. In Section 4, we solve a nonlinear projected problem. Section 5 contains a key estimate which majors the differences between $k$-th step and $ (k+1)$-th step.  We then set up  a maximization problem in Section 6. Finally in Section 7, we show that the solution to the maximization problem is indeed a solution of (\ref{p}) and prove Theorem \ref{theo1}.

\medskip

Throughout this paper, unless otherwise stated, the letters $c, C$ will always denote various generic constants that are independent of $\ve$ and $k$ for $\ve$ small enough.

\bigskip

\noindent
{\bf Acknowledgment.}  Juncheng Wei was supported by a GRF grant
from RGC of Hong Kong.

\section{Notation and Some Preliminary Analysis}

In this section we introduce some notations and some preliminary analysis on approximate solutions. Our main concern is that all the estimates should be independent of $k$-the number of spikes.

Without loss of generality, we may assume that $0\in \Omega$. By the following rescaling:
\begin{equation*}
z=\ve x, \ x\in \Omega_\ve:=\{\ve z \in \Omega\},
\end{equation*}
equation (\ref{p}) becomes
\begin{equation}\label{pp}
\left\{\begin{array}{c}
\Delta u-u+u^p=0 \ \mbox{ in }\Omega_\ve\\
u>0 \mbox{ in }\Omega_\ve, \ \ \frac{\partial u}{\partial \nu}=0 \  \mbox{ on }\partial \Omega_\ve.
\end{array}
\right.
\end{equation}
For $u\in H^2 (\Omega_\ve)$, we also put
\begin{equation}
S_\ve(u)=\Delta u-u+u^p.
\end{equation}
Associated with problem (\ref{pp}) is the energy functional
\begin{equation}\label{ef}
J_\ve(u)=\frac{1}{2}\int_{\Omega_\ve}(|\nabla u|^2+u^2)-\frac{1}{p+1}\int_{\Omega_\ve}u_{+}^{p+1}, \ \ u\in H^1(\Omega_\ve),
\end{equation}
where we denote  $u_{+}=\max (u, 0)$.

Now we define the configuration space,
\begin{equation}
\Lambda_k:= \Big\{(Q_1,\cdots,Q_k)\in \Omega^k \ \Big| \ \min_{i\neq j}|Q_i-Q_j|\geq \rho\ve, \min_{i,j, d(Q_j, \partial \Omega) \leq 10 \epsilon |\ln \epsilon |  } |Q_i-Q_j^*|\geq \rho\ve\Big\},
\end{equation}
where $Q_j^*=Q_j+2d(Q_j,\partial \Omega)\nu_{\bar{Q}_j}$, $\nu_{\bar{Q}_j}$ denotes the unit outer normal at $\bar{Q}_j \in \partial \Omega$, and $\bar{Q}_j$ is the unique point on $\partial \Omega$ such that $d(Q_j,\partial \Omega)=d(Q_j,\bar{Q}_j)$, and $\rho$ is a constant which is large enough (but independent of $\ve$).  (This is possible since $ d(Q_j, \partial \Omega) \leq 10 \epsilon | \ln \epsilon |$.)

By the definition above, we may assume that
\begin{equation}
1\leq k\leq \frac{\delta}{\ve^n\rho^n}
\end{equation}
for some $\delta>0$ sufficiently small only depend on $\Omega$ ,$n$ and $p$. We can get a lower bound of $\rho$, so we have a upper bound of $k$ which is of $O(\frac{1}{\ve^n})$. See Remark \ref{remark501} below.

Let $w$ be the unique solution of
\begin{equation}
\left\{\begin{array}{c}
\Delta w-w+w^p=0, \ w>0 \mbox{ in }\R^n,\\
 w(0)=\max_{y\in \R^n}w(y), \ w\to 0 \mbox{ as } |y|\to \infty.
 \end{array}
 \right.
\end{equation}

By the well-known result of Gidas, Ni and Nirenberg \cite{GNN}, $w$ is radially symmetric and is strictly decreasing, and $w'(r)<0$ for $r>0$. Moreover, we have the following asymptotic behavior of $w$:
\begin{equation}\label{e201}
\left\{\begin{array}{c}
w(r)=A_nr^{-\frac{n-1}{2}}e^{-r}(1+O(\frac{1}{r}))\\
w'(r)=-A_nr^{-\frac{n-1}{2}}e^{-r}(1+O(\frac{1}{r}))
\end{array}
\right.
\end{equation}
for $r>0$ large, where $A_n$ is a positive constant.

Let $K(r)$ be the fundamental solution of $-\Delta+1$ centered at $0$. Then we have
\begin{equation}\label{e202}
\left\{\begin{array}{c}
w(r)=(A_0+O(\frac{1}{r}))K(r)\\
w'(r)=-(A_0+O(\frac{1}{r}))K(r)
\end{array}
\right.
\end{equation}
for $r>0$ large, where $A_0$ is a positive constant.

For $Q\in \Omega$, we define $w_{\ve,Q}$ to be the unique solution of
\begin{equation}
\Delta v-v+w(\cdot-\frac{Q}{\ve})^p=0 \ \mbox{ in } \Omega_\ve, \ \frac{\partial v}{\partial \nu}=0 \ \mbox{ on }\partial \Omega_\ve.
\end{equation}

We first analyze $w_{\ve,Q}$. To this end, set
\begin{equation}
\varphi_{\ve,Q}=w(\frac{z-Q}{\ve})-w_{\ve,Q}(\frac{z}{\ve}).
\end{equation}

We state the following  lemma on the properties of $\varphi_{\ve,Q}$:
\begin{lemma}
Assume that $c\ve\leq d(Q,\partial \Omega)\leq 10 \ve|\ln\ve|$,
where $c\geq \frac{\rho}{2}$. We have
\begin{equation}
\varphi_{\ve,Q}=-(A_0+O(\frac{1}{\rho^{\frac{1}{2}}}))K(\frac{z-Q^*}{\ve})+O(e^{-2\rho}).
\end{equation}
\end{lemma}
\begin{proof}

In Lemma 2.1 of \cite{LNW}, a similar estimate was proved under the condition that $C_1 \ve |\ln \ve | \leq d(Q,\partial \Omega)\leq \delta$. Here we will relax this condition to $c\ve\leq d(Q,\partial \Omega)\leq 10 \ve|\ln\ve|$. The proof is similar. For the sake of completeness, we repeat a modification of the proof here.

Let $\psi_\ve(z)$ be the unique solution of
\begin{equation}
\ve^2\Delta \psi_\ve-\psi_\ve=0 \ \mbox{ in }\Omega,\ \frac{\partial \psi_\ve}{\partial \nu}=0 \ \mbox{ on }\partial \Omega.
\end{equation}
It is easy to see that
\begin{equation}
0<\psi_\ve(z)\leq \psi_1(z)\leq C \ \mbox{ for }\ve<1.
\end{equation}
On the other hand, $\varphi_{\ve,Q}$ satisfies
\begin{equation}
\ve^2\Delta v-v=0 \ \mbox{ in }\Omega,\  \frac{\partial v}{\partial \nu}=\frac{\partial}{\partial \nu}w(\frac{z-Q}{\ve})\  \mbox{ on }\partial \Omega.
\end{equation}
Using (\ref{e201}), we can see that on $\partial \Omega$,
\begin{eqnarray*}
\frac{\partial}{\partial \nu}w(\frac{z-Q}{\ve})&=&\frac{1}{\ve}w'(\frac{z-Q}{\ve})\frac{\langle z-Q,\nu\rangle}{|z-Q|}\\
&=&-(A_n+O(\frac{1}{\rho}))\ve^{\frac{n-3}{2}}
|z-Q|^{-\frac{n+1}{2}}e^{-\frac{|z-Q|}{\ve}}\langle z-Q,\nu\rangle.
\end{eqnarray*}

We use the following comparison function:
\begin{equation}
\varphi_1(z)=-(A_0-\frac{1}{\rho^{\frac{1}{2}}})K(\frac{z-Q^*}{\ve})
+e^{-d\rho}\psi_\ve,
\end{equation}
where $d\geq 2$ is a constant.

For $z\in \partial \Omega$, $|z-Q|\geq \ve^{\frac{3}{4}}$, we have
\begin{eqnarray*}
\frac{\partial \varphi_1(z)}{\partial \nu}=-(A_0-\frac{1}{\rho^{\frac{1}{2}}})K'(\frac{z-Q^*}{\ve})\ve^{-1}\frac{\langle z-Q^*,\nu\rangle}{|z-Q^*|}+e^{-d\rho}\geq \frac{1}{2}e^{-d\rho},
\end{eqnarray*}
\begin{eqnarray*}
\frac{\partial \varphi_{\ve,Q}}{\partial \nu}\leq ce^{-\ve^{-\frac{1}{4}}},
\end{eqnarray*}
so
\begin{eqnarray*}
\frac{\partial \varphi_{\ve,Q}}{\partial \nu}\leq \frac{\partial \varphi_1}{\partial \nu}.
\end{eqnarray*}

For $|z-Q|\leq \ve^{\frac{3}{4}}$, we have
\begin{eqnarray*}
\frac{\partial \varphi_1}{\partial \nu}=\frac{\partial }{\partial \nu}\{-(A_0-\frac{1}{\rho^{\frac{1}{2}}})K(\frac{z-Q^*}{\ve})\}+e^{-d\rho}.
\end{eqnarray*}
Since
\begin{eqnarray*}
&&\frac{\langle z-Q,\nu\rangle}{|z-Q|}=-(1+O(\ve^{\frac{1}{2}}))\frac{\langle z-Q^*,\nu\rangle}{|z-Q^*|},\\
&&\frac{|z-Q|}{\ve}=(1+O(\ve^{\frac{1}{2}}))\frac{|z-Q^*|}{\ve},
\end{eqnarray*}
 we obtain
\begin{eqnarray*}
\frac{\partial \varphi_{\ve,Q}}{\partial \nu}\leq \frac{\partial \varphi_1}{\partial \nu}.
\end{eqnarray*}
By the comparison principle, we have
\begin{equation}
\varphi_{\ve,Q}(z)\leq \varphi_1(z) \mbox{ for }z\in \Omega.
\end{equation}
Similarly, we obtain
\begin{equation}
\varphi_{\ve,Q}\geq {-(A_0+\frac{1}{\rho^{\frac{1}{2}}})K(\frac{z-Q^*}{\ve})-e^{-d\rho}\psi_\ve} \mbox{ for }z\in \Omega.
\end{equation}
\end{proof}

For $\mathbf{Q}=(Q_1,\cdots,Q_k)\in \Lambda_k$, we define
\begin{equation}
w_{Q_i}(x)=w(x-\frac{Q_i}{\ve}),\ w_{\ve,\mathbf{Q}}=\sum_{i=1}^k w_{\ve,Q_i}.
\end{equation}
The next lemma analyzes $w_{\ve,\mathbf{Q}}$ in $\Omega_\ve$. To this end, we divide $\Omega_\ve$ into $k+1$ parts:
\begin{equation}
\Omega_{\ve,i}=\{|x-\frac{Q_i}{\ve}|\leq \frac{\rho}{2}\},\ i=1,\cdots,k,
\end{equation}
\begin{equation}
\Omega_{\ve,k+1}=\Omega_{\ve} \backslash \cup_{i=1}^k\Omega_{\ve,i}.
\end{equation}
Then we have the following lemma
\begin{lemma}\label{lemma2}
For $x\in \Omega_{\ve,i}$,  $i=1,\cdots,k$, we have
\begin{equation}\label{in}
w_{\ve,\mathbf{Q}}=w_{\ve,Q_i}+O(e^{-\frac{\rho}{2}}).
\end{equation}
For $x\in \Omega_{\ve,k+1}$, we have
\begin{equation}\label{out}
w_{\ve,\mathbf{Q}}=O(e^{-\frac{\rho}{2}}).
\end{equation}
\end{lemma}
\begin{proof}
For $j\neq i$, and $x\in \Omega_{\ve,i}$, we have
\begin{eqnarray*}
w_{\ve,Q_j}&=&w(x-\frac{Q_j}{\ve})-\varphi_{\ve,Q_j}(\ve x)\\
&=& O(e^{-|x-\frac{Q_j}{\ve}|}+e^{-|x-\frac{Q_j^*}{\ve}|})\\
&=& O(e^{-|x-\frac{Q_j}{\ve}|})
\end{eqnarray*}
by the definition of the configuration set. Next we observe that given a a ball of size $\rho$, there are at most $c_n:=6^n$ number of non-overlapping balls of size $\rho$ surrounding this ball.  Thus  we have  for  $x\in \Omega_{\ve,i}$,
\begin{eqnarray*}
\sum_{j\neq i}w_{\ve,Q_j} (x) &=&O(\sum_{j\neq i}e^{-|x-\frac{Q_j}{\ve}|})+O(e^{-2\rho})\\
& \leq & c_n e^{-\frac{\rho}{2}}+c_n^2e^{-\rho}+\cdots+c_n^j e^{-\frac{j\rho}{2}}+\cdots \\
&\leq & \sum_{j=1}^\infty e^{j(\log c_n-\frac{\rho}{2})}+O(e^{-2\rho})\\
& \leq &O( e^{-(\frac{\rho}{2}-\log c_n)})\\
& \leq & O(e^{-\frac{\rho}{2}}),
\end{eqnarray*}
if $c_n<e^{\frac{\rho}{2}}$, which is true for $\rho$ large enough.  So this proves (\ref{in}). The proof of (\ref{out}) is similar.

\end{proof}

\medskip

\noindent
 \begin{remark}
 In the following sections, we will use the definition of the configuration and the estimate as above frequently.
\end{remark}

The following lemma is proved in Lemma 2.3 of \cite{BL}.

\begin{lemma}\label{lemma22}
Let $f\in C(\R^n)\cap L^\infty(\R^n)$, $g\in C(\R^n)$ be radially symmetric and satisfy for some $\alpha\geq 0$, $\beta\geq 0$, $\gamma_0\in \R$,
\begin{eqnarray*}
&&f(x)exp(\alpha|x|)|x|^\beta\to \gamma_0 \mbox{ as } |x|\to \infty,\\
&&\int_{\R^n}|g(x)|exp(\alpha|x|)(1+|x|^\beta)dx<\infty.
\end{eqnarray*}
Then
\begin{equation*}
exp(\alpha|y|)|y|^\beta\int_{\R^n}g(x+y)f(x)dx\to \gamma_0\int_{\R^n}g(x)exp(-\alpha x_1)dx
\mbox{ as }|y|\to \infty .
\end{equation*}
\end{lemma}

As in \cite{LNW}, we now define the following quantities:
\begin{equation}
B_\ve(Q_j)=-\int_{\Omega_\ve}w_{Q_j}^p\varphi_{\ve,Q_j}dx, \ \ B_{\ve}(Q_i,Q_j)=\int_{\Omega_\ve}w_{Q_i}^pw_{Q_j}dx.
\end{equation}
Then we have the following:
\begin{lemma}\label{lemma3}
For $\mathbf{Q}=(Q_1,\cdots,Q_k)\in \Lambda_k$, it holds that
\begin{equation}
B_\ve(Q_j)=(\gamma+O(\frac{1}{\sqrt{\rho}}))w(\frac{2d(Q_j,\partial \Omega)}{\ve})+O(e^{-(1+\xi)\rho}),
\end{equation}
\begin{equation}
B_\ve(Q_i,Q_j)=(\gamma+O(\frac{1}{\sqrt{\rho}}))w(\frac{2d(Q_j,\partial \Omega)}{\ve})+O(e^{-(1+\xi)\rho}),
\end{equation}
for some $\xi>0$ independent of $\ve$ and $k$ for $\ve$ sufficiently small,
where
\begin{equation}
\label{gammadef}
\gamma=\int_{\R^n}w^p(y)e^{-y_1}dy.
\end{equation}
\end{lemma}

\begin{remark}
Note that $\gamma>0$. See Lemma 4.7 of \cite{NW}.
\end{remark}

\begin{proof}
By Lemma \ref{lemma2} and \ref{lemma22}, the proof is similar to
 that of Lemma 2.5 in \cite{LNW}. We omit the details.
\end{proof}

\section{Linear Theory}\label{sec3}
In this section, we study a linear theory that allow us to perform the finite dimensional reduction procedure.  The proof is similar to Section 3 of \cite{LNW}. However, the main concern is to show that all the constants are independent of the number $k$.  Fixing an integer $k$ satisfying
\begin{equation}
1 \leq k \leq \frac{\delta }{\ve^n}
\end{equation}
and  $\mathbf{Q}\in \Lambda_k$, we define the following functions:
\begin{equation}\label{zij}
Z_{ij}=\frac{\partial w_{Q_i}}{\partial x_j}\chi_i(x), \mbox{ for } i=1,\cdots,k, \ j=1,\cdots,n,
\end{equation}
where $w_{Q_i}(x)=w(x-\frac{Q_i}{\ve})$,
$\chi_i(x)=\chi(\frac{2|\ve x-Q_i|}{(\rho-1)\ve})$ and $\chi(t)$ is a cut off function such that $\chi(t)=1$ for $|t|\leq 1$ and $\chi(t)=0$ for $|t|\geq \frac{\rho^2}{\rho^2-1}$. Note that the support of $Z_{ij}$ belongs to $B_{\frac{\rho^2}{2(\rho+1)}}(\frac{Q_i}{\ve})$.

We consider the following linear problem: Given $h$, find a function $\phi$ satisfying
\begin{equation}\label{e301}
\left\{ \begin{array}{ll}
L (\phi ) :=\Delta \phi-\phi+pw^{p-1}_{\ve,\mathbf{Q}}\phi=h+\sum_{i=1,\cdots,k,j=1,\cdots,n}   c_{ij} Z_{ij} \quad {\mbox {in}} \quad \Omega_\ve \\
\\
 {\partial \phi \over \partial \nu} = 0  \quad {\mbox {on}} \quad \partial \Omega_\ve \\
 \\
 \int_{\Omega_\ve} \phi Z_{ij}  =0 \quad
 {\mbox{for}} \quad i=1,\cdots,k,\ j=1,\cdots,n.
 \end{array} \right.
\end{equation}

Let
\begin{equation}
\label{Wdef}
W : =  \sum_{\mathbf{Q} \in \Lambda_k} e^{-\eta \, |\cdot - \frac{Q_i}{\ve} |}.
\end{equation}
Given $0<\eta < 1$, consider the norm
\begin{equation}\label{e302}
 \quad \| h \|_{*} =\sup_{x \in \Omega_\ve} |   W(x)^{-1}  h(x) |
\end{equation}
where $(Q_1,\cdots,Q_k) \in \Lambda_k$.

\begin{proposition} \label{p301}
There exist positive numbers $\eta \in (0,1)$,
$\ve_0>0$, $\rho_0>0$ and $C>0$, such that for all $0< \ve < \ve_0$, $\rho>\rho_0$, and for any given $h$ with $\|h\|_*$ norm bounded,
there is a unique solution $(\phi ,\{ c_{ij}\}  )$  to problem
\equ{e301}. Furthermore
\begin{equation}\label{apest}
 \|\phi \|_{*} \le C \|h\|_*  .
\end{equation}
\end{proposition}

The proof of the above Proposition, which we postpone to the end of
this section, is based on Fredholm alternative Theorem for compact
operator and an a-priori bound for solution to \equ{e301} that we
state (and prove) next.

\begin{proposition} \label{p302}
 Let $h$ with $\| h \|_* $ bounded and
assume that $(\phi , \{c_{ij} \} )$ is a solution to \equ{e301}.
Then there exist positive numbers $\ve_0$, $\rho_0$ and $C$, such that for all $0<\ve <\ve_0$, $\rho>\rho_0$ and $\mathbf{Q}\in \Lambda_k$, one has
\begin{equation}\label{apest1}
 \|\phi \|_{*} \leq C \|h\|_*  ,
\end{equation}
where $C$ is a positive constant independent of $\ve$, $\rho$, $k$ and $\mathbf{Q}\in \Lambda_k$.
\end{proposition}

\medskip
\begin{proof}
We argue by contradiction. Assume that there exist $\phi $ solution
to \equ{e301} and
$$
\| h \|_* \to 0, \quad \| \phi \|_{*} =1.
$$
Multiplying the equation in \equ{e301} against $Z_{ij}$ and integrating in
$\Omega_\ve$, we get
$$
\int_{\Omega_\ve} L \phi Z_{ij} (x)  = \int_{\Omega_\ve}  h Z_{ij}  + c_{ij} \int_{\Omega_\ve} Z^2_{ij}.
$$
 Given the exponential decay at infinity of $\partial_{x_i} w $ and the definition of $Z_{ij} $, we get
\begin{equation}
\int_{\Omega_\ve}  Z^2_{ij} =  \int_{\R^n} w_{x_i}^2 + O(e^{-\delta_1 \rho}) , \label{dent1} \mbox{ as }\rho\to\infty,
\end{equation}
for some $\delta_1 >0$. On the other hand
$$
|\int_{\Omega_\ve}  h Z_{ij} | \leq C \| h \|_* \int_{\Omega_\ve} | w_{x_i}
(x-\frac{Q_i}{\ve}) |  e^{-\eta |x-\frac{Q_i}{\ve}|} dx\leq C \| h \|_* .
$$
Here and in what follows, $C$ stands for a positive constant
independent of $\ve$, and $\rho$, as $\ve \to 0$, $\rho\to \infty$. Now if we write
$\tilde Z_{ij} (x)=  w_{x_i} (x-\frac{Q_i}{\ve})$, we have
\begin{eqnarray}
-\int_{\Omega_\ve} L \phi Z_{ij} (x) &=& -\int_{\Omega_\ve}  \phi  (L [Z_{ij} ]) \nonumber \\
 & = &   \int_{B(\frac{Q_i}{\ve}, {\rho \over 2})} [\Delta \tilde
Z_{ij} - \tilde Z_{ij}
 + pw^{p-1}(x-\frac{Q_i}{\ve}) \tilde Z_{ij} ] \chi_i \phi \nonumber \\
\nonumber \\
&-& \int_{B(z,{\rho \over 2} )} \phi ( \tilde Z_{ij} \Delta \chi_i+ 2 \nabla \chi_i\nabla Z_{ij} ) \\
&+& p\int_{B(\frac{Q_i}{\ve}, {\rho\over 2})} (w^{p-1}_{\ve,\mathbf{Q}}-w^{p-1}(x-\frac{Q_i}{\ve})) \phi \tilde Z_{ij} \chi_i.\nonumber
\label{dent0}
\end{eqnarray}

Next we estimate all the terms in the above equation.

The first term is $0$ since
\begin{equation*}
\Delta \tilde Z_{ij} - \tilde Z_{ij}
 +pw^{p-1}(x-\frac{Q_i}{\ve}) \tilde Z_{ij}=0.
\end{equation*}

The second integral can be estimated as
follows
\begin{eqnarray*}
\left| \int_{\Omega_\ve} \phi ( \tilde Z_{ij} \Delta  \chi_i+ 2 \nabla \chi_i\nabla \tilde Z_{ij} ) \right|
&&\leq C\|\phi\|_*\int_{\frac{\rho-1}{2}}^{\frac{\rho^2}{2(\rho+1)}}e^{-(1+\eta)s}s^{\frac{-(n-1)}{2}}ds\\
&&\leq C e^{-(1+\xi) {\rho \over 2}} \| \phi \|_{*},
\end{eqnarray*}
for some $\xi >0$. Finally, we observe that in  $B(\frac{Q_i}{\ve},{\rho \over 2} )$ the following holds
$$
| w^{p-1}_{\ve,\mathbf{Q}}-w^{p-1}_{Q_i}(x)| \leq Cw_{Q_i}^{p-2}(x)
\left[ \sum_{j\neq i } w (x-\frac{Q_j}{\ve} ) \right].
$$
Thus  we obtain
\begin{eqnarray*}
\left| \int_{B(\frac{Q_i}{\ve}, {\rho \over 2})} (w^{p-1}_{\ve,\mathbf{Q}}-w^{p-1}_{Q_i}(x) ) \phi \tilde Z_{ij} \chi_i
\right|
&&
\leq  C e^{-\xi {\rho \over 2}} \| \phi \|_{*}
\end{eqnarray*}
for some  $\xi >0$, depending on $n$ and $p$. We then conclude
that
\begin{equation}\label{cijl}
| c_{ij}| \leq C \left[ e^{-\xi {\rho \over 2}} \|
\phi \|_{*} + \| h \|_* \right].
\end{equation}

\medskip
Let now $\eta \in (0,1)$. It is easy to check that the function $W$ (defined at (\ref{Wdef})) satisfies
\[
L \, W  \leq  \frac{1}{2} \,  ( \eta^2 -1) \, W \, ,
\]
in $\Omega_\ve \setminus \cup_{i=1}^k B(\frac{Q_i}{\ve} , \rho_1)$ provided $\rho_1$ is large enough but independent of $\rho$.  Hence the function $W$ can be used as a barrier to prove the pointwise estimate
\begin{equation}
|\phi | (x)  \leq  C \, \left( \|  L  \, \phi \|_*  + \sup_i \|\phi\|_{L^\infty(\partial B(\frac{Q_i}{\ve},\rho_1))} \right) \,   W (x) \, ,
\label{eq:fp2}
\end{equation}
for all $x \in \Omega_\ve \setminus \cup_{i=1}^k B(\frac{Q_i}{\ve} , \rho_1)$.

\medskip

Granted these preliminary estimates, the proof  of the result goes by contradiction. Let us assume there exist  a sequence of $\ve$ tending to $0$, $\rho$ tending to $\infty$ and a sequence of solutions of  \equ{e301}  for which the inequality is not true. The problem being linear, we can reduce to the case where we have a sequence $\ve^{(n)}$ tending to $0$, $\rho^{(n)}$ tending to $\infty$ and sequences $ h^{(n)}$, $\phi^{(n)}, \{c_{ij}^{(n)}\}$ such that
$$
\| h^{(n)} \|_* \to 0, \quad \mbox{and} \quad \| \phi^{(n)} \|_{*} =1.
$$
By (\ref{cijl}), we can get that
\[
\| \sum_{ij} c_{ij}^{(n)} Z_{ij} \|_* \to 0 \, .
\]
Then \equ{eq:fp2} implies that there exists $Q_i^{(n)} \in \Lambda_k$  such that
\begin{equation}
\|\phi^{(n)} \|_{L^\infty (B(Q_i^{(n)} ,\frac{\rho}{2}))}\geq C \, ,
\label{eq:fp3}
\end{equation}
for some fixed constant $C>0$. Using elliptic estimates together with Ascoli-Arzela's theorem, we can find a sequence $Q_i^{(n)}$ and we can extract, from the sequence $\phi^{(n)} (\cdot -\frac{Q_i^{(n)}}{\ve})$ a subsequence which will converge (on compact sets) to $\phi_\infty$ a solution of
\[
\left(\Delta - 1 + pw^{p-1} \right) \, \phi_\infty =0 \, ,
\]
in $\R^n$, which is bounded by a constant times $e^{-\eta \, |x|}$, with $\eta >0$. Moreover, recall that $\phi^{(n)}$ satisfies the orthogonality conditions in \equ{e301}. Therefore, the limit function $\phi_\infty$ also satisfies
\[
\int_{\R^n} \phi_\infty \, \nabla w \, dx =0 \, .
\]
By the nondegeneracy of solution $w$, we have  that $\phi_\infty \equiv 0$, which is certainly in contradiction with \equ{eq:fp3} which implies that $\phi_\infty$ is not identically equal to $0$.

\medskip

Having reached a contradiction, this completes the proof of the
Proposition.
\qed

\medskip
We can now prove Proposition \ref{p301}.

\medskip
\noindent {\it Proof of Proposition \ref{p301}.} Consider the space
$$
{\mathcal H} = \{ u \in H^1 (\Omega_\ve ) \, : \, \int_{\Omega_\ve} u Z_{ij} = 0 , \quad  (Q_1,\cdots,Q_k)\in \Lambda_k
\}.
$$
Notice that the problem \equ{e301} in $\phi $ gets re-written as
\begin{equation}\label{lp7}
\phi + K (\phi ) = \bar h \quad {\mbox{in}} \quad {\mathcal H}
\end{equation}
where $\bar h$ is defined by duality and $K: {\mathcal H} \to
{\mathcal H}$ is a linear compact operator. Using Fredholm's
alternative, showing that equation \equ{lp7} has a unique solution
for each $\bar h$ is equivalent to showing that the equation has a
unique solution for $\bar h = 0$, which in turn follows from
Proposition \ref{p302}. The estimate \equ{apest} follows directly from
Proposition \ref{p302}. This concludes the proof of Proposition
\ref{p301}.

\medskip
In the following, if $\phi$ is the unique solution given by
Proposition \ref{p301}, we set
\begin{equation}
\label{carmen1} \phi = {\mathcal A} (h).
\end{equation}
Estimate \equ{apest} implies
\begin{equation}\label{carmen2}
\| {\mathcal A} (h ) \|_{*} \leq C \| h \|_{*}.
\end{equation}

\setcounter{equation}{0}
\section{The non linear projected problem}\label{sec4}
For small $\ve$, large $\rho$, and fixed points $\mathbf{Q}\in \Lambda_k$, we show solvability in $\phi$, $\{c_{ij}\}$  of
the non linear projected problem

\begin{equation}\label{e401}
\left\{\begin{array}{c}
\Delta (w_{\ve,\mathbf{Q}}+\phi)-(w_{\ve,\mathbf{Q}}+\phi)+(w_{\ve,\mathbf{Q}}+\phi)^p
=\sum_{i=1,\cdots,k,j=1,\cdots,n}c_{ij} Z_{ij}  \quad {\mbox {in}} \quad \Omega_\ve \\
\\
 {\partial \phi \over \partial \nu} = 0 , \quad {\mbox {on}} \quad \partial \Omega_\ve \\
 \\
 \int_{\Omega_\ve } \phi Z_{ij} =  0 \quad
 {\mbox{ for }} i=1,\cdots,k,\ j=1,\cdots,n .
 \end{array} \right.
 \end{equation}

The first equation in (\ref{e401}) can be rewritten as
\begin{equation}\label{e402}
L (\phi ):=\Delta \phi-\phi+pw_{\ve,\mathbf{Q}}^{p-1}\phi = S_\ve(w_{\ve,\mathbf{Q}})+N(\phi) +
\sum_{i=1,\cdots,k,j=1,\cdots,n} c_{ij} Z_{ij},
\end{equation}
where
\begin{eqnarray}
&&S_\ve(w_{\ve,\mathbf{Q}})=\Delta w_{\ve,\mathbf{Q}}-w_{\ve,\mathbf{Q}}+w_{\ve,\mathbf{Q}}^p,\\
&&N(\phi)=(w_{\ve,\mathbf{Q}}+\phi)^p-w_{\ve,\mathbf{Q}}^p-pw_{\ve,\mathbf{Q}}^{p-1}\phi.
\end{eqnarray}

We have the validity of the following result:

\begin{proposition} \label{p401}
 There exist positive numbers $\ve_0$, $\rho_0$, $C$ and $\xi >0$ such that for all $\ve \leq \ve_0$, $\rho\geq \rho_0$, and for any $\mathbf{Q}\in\Lambda_k$, there is a unique solution $(\phi_{\ve,\mathbf{Q}} , \{c_{ij}\}  )$  to problem
\equ{e401}. Furthermore $\phi_{\ve,\mathbf{Q}}$ is $C^1$ in $\mathbf{Q}$ and we have
\begin{equation}
 \|\phi_{\ve,\mathbf{Q}} \|_{*} \leq C e^{-{(1+\xi)   \over 2} \, \rho }.
\label{est2}\end{equation}

\end{proposition}

\medskip
\noindent
{\it Proof.} The proof relies on the contraction mapping theorem in
the $\| \cdot \|_{*}$-norm introduced above. Observe that $\phi$
solves \equ{e401} if and only if
\begin{equation}\label{fixed}
\phi = {\mathcal A} \left( S_\ve(w_{\ve,\mathbf{Q}})+N(\phi ) \right)
\end{equation}
where ${\mathcal A}$ is the operator introduced in \equ{carmen1}. In
other words, $\phi$ solves \equ{e401} if and only if $\phi$ is a
fixed point for the operator
$$
T(\phi ) :={\mathcal A} \left( S_\ve(w_{\ve,\mathbf{Q}})+N(\phi ) \right).
$$
Given $r>0$, define
$$
{\mathcal B} = \{ \phi \in C^2 (\Omega_\ve ) \, : \, \| \phi \|_{*} \leq
r e^{-{(1+\xi) \over 2} \rho} , \, \int_{\Omega_\ve} \phi Z_{ij}
= 0 \}.
$$
We will prove that $T$ is a contraction mapping from ${\mathcal B}$
in itself.

To do so, we claim that
\begin{equation}
\label{estimateE} \| S_\ve(w_{\ve,\mathbf{Q}}) \|_{*} \leq C e^{-{(1+\xi ) \over 2} \rho}
\end{equation}
and
\begin{equation}
\label{estimateN} \| N(\phi ) \|_{*} \leq C \left[ \| \phi \|_{*}^2
+ \| \phi \|_{*}^p\right] ,
\end{equation}
for some fixed function $C$ independent of $\rho$ and $\ve$. We postpone the proof of the estimates above to the end of the proof of this Proposition. Assuming the validity of
\equ{estimateE} and \equ{estimateN} and taking into account
\equ{carmen2}, we have for any $\phi \in {\mathcal B}$
$$
\begin{array}{ll}
\| T( \phi ) \|_{*} \leq & C\left[ \| S_\ve(w_{\ve,\mathbf{Q}})+N(\phi ) \|_* \right] \leq
C \left[ e^{-{(1+\xi ) \over 2} \rho} + r^2 e^{-{(1+\xi )} \rho} +
r^p
e^{-{p (1+\xi ) \over 2} \rho}\right]\\
\\
&\leq r e^{-{(1+\xi ) \over 2} \rho}
\end{array}
$$
for a proper choice of $r$ in the definition of ${\mathcal B}$,
since $p>1$.

Take now $\phi_1 $ and $\phi_2$ in ${\mathcal B}$. Then it is
straightforward to show that
$$
\begin{array}{ll}
\| T(\phi_1 ) - T(\phi_2 )\|_{*} &\leq C \| N(\phi_1 ) - N(\phi_2 )
\|_* \\
\\
&\leq C \left[ \| \phi_1 \|_*^{\min (1, p-1)} +  \| \phi_2
\|_*^{\min (1, p-1)} \right] \, \| \phi_1  - \phi_2  \|_* \\
\\
&\leq \frac{1}{2}
\| \phi_1  - \phi_2  \|_*.
\end{array}
$$
This means that $T$ is a contraction mapping from ${\mathcal B}$
into itself.

To conclude the proof of this Proposition we are left to show the
validity of \equ{estimateE} and \equ{estimateN}. We start with
\equ{estimateE}.

Fix $Q_i\in \Lambda_k$ and consider the region $|x-\frac{Q_i}{\ve} | \leq {\rho \over
2+\sigma} $, where  $\sigma$ is a small positive number to be chosen
later. In this region the error $S_\ve(w_{\ve,\mathbf{Q}})$ can be estimated in the following way
\begin{eqnarray}
|S_\ve(w_{\ve,\mathbf{Q}})| &\leq & C \left[ w^{p-1} (x-\frac{Q_i}{\ve} ) \sum_{j\neq i}
w (x-\frac{Q_i}{\ve}) + \sum_{j\neq i } w^p (x-\frac{Q_j}{\ve} ) \right] \nonumber\\
\nonumber
\\
&\leq & C w^{p-1} (x-\frac{Q_i}{\ve} ) e^{-({1\over 2}
+ {\sigma \over 2 (2+\sigma)} ) \rho} \nonumber \\
\nonumber \\
&\leq & C w^{p-1} (x-\frac{Q_i}{\ve}) e^{-({1\over 2} + {\sigma \over 4
(2+\sigma)} ) \rho} \,  e^{- {\sigma \over 4 (2+\sigma)} \rho}
\nonumber \\
\nonumber \\
&\leq & C w^{p-1} (x-\frac{Q_i}{\ve} ) e^{-{1+\xi \over 2} \rho} \label{EE1}
\end{eqnarray}
for a proper choice of $\xi >0$.

Consider now the region $|x-\frac{Q_i}{\ve} | > {\rho \over 2+\sigma}$, for
all $i $. Since $0< \mu < p-1$, we write $\mu = p-1-M$. From
the definition of $S_\ve(w_{\ve,\mathbf{Q}})$, we get in the region under consideration
\begin{eqnarray}\label{EE2}
|S_\ve(w_{\ve,\mathbf{Q}}) |
&\leq & C \left[ \sum_{j } w^p (x-\frac{Q_j}{\ve} ) \right]
\leq C \left[ \sum_{j } e^{-\mu |x-\frac{Q_j}{\ve} |}  \right] e^{-(p-\mu)
{\rho \over 2+\sigma} }\nonumber \\
&\leq & \left[ \sum_{j} e^{-\mu |x-\frac{Q_j}{\ve}|}  \right] e^{-{1+M
\over 2+\sigma} \rho } \nonumber\\
&\leq & \left[ \sum_{j} e^{-\mu |x-\frac{Q_j}{\ve}
|}  \right] e^{-{1+\xi \over 2} \rho  }
\end{eqnarray}
for some $\xi >0$, if we chose $M$ and $\sigma$ small enough. From
\equ{EE1} and \equ{EE2} we get \equ{estimateE}.

We now prove \equ{estimateN}. Let $\phi \in {\mathcal B}$. Then
\begin{equation}
\label{enne} |N(\phi )| \leq | ( w_{\ve,\mathbf{Q}}+\phi )^p - w_{\ve,\mathbf{Q}}^p - p w_{\ve,\mathbf{Q}}^{p-1} \phi |
\leq C (\phi^2 + |\phi|^p ).
\end{equation}
Thus we have
$$
\begin{array}{ll}
| (\sum_{j} e^{-\eta |x-\frac{Q_j}{\ve}|)^{-1}  } N(\phi ) | & \leq C \| \phi \|_*
\left( |\phi| + |\phi|^{p-1} \right)\\
\\
&\leq C ( \| \phi \|_*^2 + \| \phi \|_*^p ).
\end{array}
$$
This gives \equ{estimateN}.

For the $C^1$ regularity of $\phi_{\ve,\mathbf{Q}}$, see Lemma 4.1 in \cite{LNW}. This concludes the proof of the Proposition.
\end{proof}

\section{An improved estimate}

In this section, we present a key estimate on the difference between the solutions in the $k-$th step and $ (k+1)-$th step.

For $(Q_1,\cdots,Q_{k})\in \Lambda_{k}$,  we  denote $u_{\ve, Q_1, \cdots, Q_k}$ as $ w_{\ve, Q_1,..., Q_k}+ \phi_{\ve, Q_1, ..., Q_k}$, where $ \phi_{\ve, Q_1, \cdots, Q_k}$ is the unique solution given by Proposition \ref{p401}. The estimate below says that the difference between $ u_{\ve, Q_1, \cdots, Q_{k+1}}$ and $ u_{\ve, Q_1, \cdots, Q_k}+ u_{\ve, Q_{k+1}}$ is small globally in $H^1 (\Omega_\ve)$ norm.

We now write
 \begin{eqnarray}
 u_{\ve,Q_1,\cdots,Q_{k+1}}&=&u_{\ve,Q_1,\cdots,Q_k}+u_{\ve,Q_{k+1}}+\varphi_{k+1}\\
 &=&\bar{W}+\varphi_{k+1}\nonumber,
 \end{eqnarray}
where
$$\bar{W}= u_{\ve,Q_1,\cdots,Q_k}+u_{\ve,Q_{k+1}}.$$

By Proposition \ref{p401}, we can easily derive that
\begin{equation}
\label{vark100}
 \|\varphi_{k+1} \|_{*} \leq C e^{-{(1+\xi)   \over 2} \, \rho }.
\end{equation}

However the estimate (\ref{vark100}) is not good enough.   We need the following key estimate for $\varphi_{k+1}$:

 \begin{lemma}\label{lemma501}
 Let $\rho$, $\ve$ be as in Proposition \ref{p401}. Then it holds
 \begin{equation}
\label{keyvar}
 \int_{\Omega_\ve} (|\nabla \varphi_{k+1}|^2 + \varphi_{k+1}^2 ) \leq C e^{-(1+\xi) \rho},
 \end{equation}
 for some constant $C>0,\xi>0$ independent of $\ve,\rho ,k$ and $\mathbf{Q}\in \Lambda_{k+1}$.
 \end{lemma}

 \begin{proof}
To prove (\ref{keyvar}), we need to perform a secondary decomposition.

 We first recall the following fact: it is well-known that the principal eigenfunction $\phi_0$ of the following linearized operator:
\begin{equation}
\Delta \phi-\phi+p w^{p-1}\phi=\lambda_1\phi
\end{equation}
is even and exponentially decaying, where $\lambda_1$ is the first eigenvalue. We fix $ \phi_0$ such that $ \max_{ y \in \R^n} \phi_0 =1$. Denote by $\phi_i=\chi_i\phi_0 (x-\frac{Q_i}{\ve})$, where $ \chi_i$ is the cut-off function introduced in Section 3.

By the equations satisfied by $\varphi_{k+1}$, we have
\begin{equation}\label{varphi}
\bar{L}\varphi_{k+1}=\bar{S}+\sum_{i=1,\cdots,k+1, j=1,\cdots,n}c_{ij}Z_{ij}
\end{equation}
for some constants $\{c_{ij}\}$, where
\begin{equation*}
\bar{L}=\Delta-1+p\tilde{W}^{p-1},
\end{equation*}
\begin{equation*}
\tilde{W}^{p-1}=\left\{\begin{array}{l}
\frac{ (\bar{W} + \varphi_{k+1})^p- \bar{W}^{p}}{ p \varphi_{k+1}}, \ \mbox{if} \ \varphi_{k+1} \not =0\\
\bar{W}^{p-1}, \ \mbox{if} \ \varphi_{k+1}=0,
\end{array}
\right.
\end{equation*}
and
\begin{equation*}
\bar{S}=(u_{\ve,Q_1,\cdots,Q_{k}}+u_{\ve,Q_{k+1}})^p-u_{\ve,Q_1,\cdots,Q_k}^p-u_{\ve,Q_{k+1}}^p.
\end{equation*}

The $L^2$-norm of $\bar{S}$ is estimated first: Observe that
\begin{eqnarray*}
&& |\bar{S}| = |(u_{\ve,Q_1,\cdots,Q_{k}}+u_{\ve,Q_{k+1}})^p-u_{\ve,Q_1,\cdots,Q_k}^p-u_{\ve,Q_{k+1}}^p|\\
&&\leq C(p|u_{\ve,Q_1,\cdots,Q_{k}}|^{p-1}u_{\ve,Q_{k+1}}+p|u_{\ve,Q_{k+1}}|^{p-1}u_{\ve,Q_1,\cdots,Q_k}).
\end{eqnarray*}
By the estimate in Proposition \ref{p401}, we have the following estimate of the first term above
\begin{eqnarray*}
&&\int_{\Omega_\ve}|u_{\ve,Q_1,\cdots,Q_{k}}|^{2(p-1)}u^2_{\ve,Q_{k+1}}dx\\
&&\leq C\int_{\Omega_\ve}w_{\ve,Q_1,\cdots,Q_k}^{2(p-1)}w_{\ve,Q_{k+1}}^2dx+O(e^{-(1+\xi)\rho})\\
&&\leq Ce^{-(1+\xi)\rho}.
\end{eqnarray*}
The second term can be estimated similarly. So we have
\begin{equation}\label{s}
\|\bar{S}\|_{L^2(\Omega_\ve)}\leq ce^{-(1+\xi)\frac{\rho}{2}}.
\end{equation}

By the estimate (\ref{vark100}), we have  the following estimate
\begin{equation}
\label{W123}
 \tilde{W}= \sum_{i=1}^{k+1} w(x-\frac{Q_i}{\epsilon}) + O(e^{- (1+\xi)\frac{\rho}{2}}).
\end{equation}

Decompose $\varphi_{k+1}$ as
\begin{equation}\label{decom}
\varphi_{k+1}=\psi+\sum_{i=1}^{k+1}c_i\phi_i
+\sum_{i=1,\cdots,k+1,j=1,\cdots,n}d_{ij}Z_{ij}
\end{equation}
for some $c_i,d_{ij}$ such that
\begin{equation}
\label{345}
\int_{\Omega_\ve} \psi \bar{L}\phi_idx=\int_{\Omega_\ve}\psi Z_{ij}dx=0,\ i=1,..., k, \ j=1,..., n.
\end{equation}

Since
\begin{equation}
\varphi_{k+1}=\phi_{\ve,Q_1,\cdots,Q_{k+1}}-\phi_{\ve,Q_1,\cdots,Q_k}-\phi_{\ve,Q_{k+1}},
\end{equation}

we have for $i=1,\cdots,k$,
\begin{eqnarray*}
d_{ij}&=&\int_{\Omega_\ve} \varphi_{k+1}Z_{ij}\\
&=&\int_{\Omega_\ve}(\phi_{\ve,Q_1,\cdots,Q_{k+1}}-\phi_{\ve,Q_1,\cdots,Q_k}-\phi_{\ve,Q_{k+1}})Z_{ij}\\
&=&-\int_{\Omega_\ve} \phi_{\ve,Q_{k+1}}Z_{ij}
\end{eqnarray*}
and
\begin{eqnarray*}
d_{k+1,j}&=&\int_{\Omega_\ve} \varphi_{k+1}Z_{k+1,j}\\
&=&\int_{\Omega_\ve}(\phi_{\ve,Q_1,\cdots,Q_{k+1}}-\phi_{\ve,Q_1,\cdots,Q_k}-\phi_{\ve,Q_{k+1}})Z_{k+1,j}\\
&=&-\int_{\Omega_\ve} \phi_{\ve,Q_1,\cdots,Q_k}Z_{k+1,j},
\end{eqnarray*}
where we use the orthogonality conditions satisfied by $\phi_{\ve,Q_1,\cdots,Q_k}$ and $\phi_{\ve,Q_{k+1}}$.
So by Proposition \ref{p401}, we have
\begin{equation}\label{d}
\left\{\begin{array}{ll}
|d_{ij}|\leq ce^{-(1+\xi)\frac{\rho}{2}}e^{-\eta\frac{|Q_i-Q_{k+1}|}{\ve}} \mbox{ for } i=1,\cdots,k\\
|d_{k+1,j}|\leq ce^{-(1+\xi)\frac{\rho}{2}}\sum_{i=1}^ke^{-\eta\frac{|Q_i-Q_{k+1}|}{\ve}}
\end{array}
\right.
\end{equation}
for some $\eta>0$.

By (\ref{decom}), we can rewrite (\ref{varphi}) as
\begin{equation}\label{decom1}
\bar{L}\psi+\sum_{i=1}^{k+1} c_i\bar{L}\phi_i+\sum_{i=1,\cdots,k+1,j=1,\cdots,n}d_{ij}\bar{L}Z_{ij}
=\bar{S}+\sum_{i=1,\cdots,k+1,j=1,\cdots,n}c_{ij}Z_{ij}.
\end{equation}

To obtain the estimates for the coefficients $c_i$ , we use the equation (\ref{decom1}).

First, multiplying (\ref{decom1}) by $\phi_i$ and integrating over $\Omega_\ve$, we have
\begin{eqnarray}
\label{W345}
c_i\int_{\Omega_\ve} \bar{L}(\phi_i)\phi_i=-\sum_{j=1}^n d_{ij}\int_{\Omega_\ve} \bar{L}(Z_{ij})\phi_i+\int_{\Omega_\ve}\bar{S}\phi_i
\end{eqnarray}
where
\begin{equation}
\left\{\begin{array}{ll}
\label{W346}
|\int_{\Omega_\ve}\bar{S}\phi_i|\leq ce^{-(1+\xi)\frac{\rho}{2}} e^{-\eta\frac{|Q_i-Q_{k+1}|}{\ve}} \mbox{ for }i=1,\cdots,k\\
|\int_{\Omega_\ve}\bar{S}\phi_{k+1}|\leq ce^{-(1+\xi)\frac{\rho}{2}}\sum_{i=1}^k e^{-\eta\frac{|Q_i-Q_{k+1}|}{\ve}}.
\end{array}
\right.
\end{equation}

 From (\ref{W123}) we see that
\begin{equation}
\label{W234}
\int_{\Omega_\ve}\bar{L}(\phi_i)\phi_i =  - \lambda_1 \int_{\R^n} \phi_0^2 + O(e^{-(1+\xi)\frac{\rho}{2}}).
\end{equation}

Combining (\ref{d}) and (\ref{W345})-(\ref{W234}), we have
\begin{equation}\label{c}
\left\{\begin{array}{ll}
|c_i|\leq  ce^{-(1+\xi)\frac{\rho}{2}} e^{-\eta\frac{|Q_i-Q_{k+1}|}{\ve}}, \ i=1,..., k\\
|c_{k+1}|\leq  ce^{-(1+\xi)\frac{\rho}{2}}\sum_{i=1}^k e^{-\eta\frac{|Q_i-Q_{k+1}|}{\ve}}.
\end{array}
\right.
\end{equation}


Next let us estimate $\psi$. Multiplying (\ref{decom1}) by $\psi$ and integrating over $\Omega_\ve$, we find
\begin{eqnarray}\label{psi}
\int_{\Omega_\ve} \bar{L}(\psi)\psi=\int_{\Omega_\ve}\bar{S} \psi-\sum_{i=1,\cdots,k+1,j=1,\cdots,n}d_{ij}\int_{\Omega_\ve} \bar{L}(Z_{ij})\psi.
\end{eqnarray}
We claim that
\begin{equation}
\int_{\Omega_\ve} [-\bar{L}(\psi)\psi] \geq c_0 \|\psi\|^2_{H^1(\Omega_\ve)}
\end{equation}
for some constant $c_0>0$.

Since the approximate solution is exponentially decaying away from the points $\frac{Q_i}{\ve}$, we have
\begin{equation}
\int_{\Omega_\ve \backslash  \cup_i B_{\frac{\rho-1}{2}}(\frac{Q_i}{\ve})}\bar{L}(\psi)\psi\geq \frac{1}{2}
\int_{\Omega_\ve \backslash \cup_i B_{\frac{\rho-1}{2}}(\frac{Q_i}{\ve})}|\nabla \psi|^2+|\psi|^2.
\end{equation}
Now we only need to prove the above estimates in the domain $\cup_i B_{\frac{\rho-1}{2}}(\frac{Q_i}{\ve})$. We prove it by contradiction. Otherwise, there exists a sequence $\rho_n\to +\infty$, and $Q_i^{(n)}$ such that
\begin{eqnarray*}
\int_{B_{\frac{\rho_n-1}{2}}(\frac{Q_i^{(n)}}{\ve})}|\nabla \psi_n|^2+|\psi_n|^2=1,\ \int_{B_{\frac{\rho_n-1}{2}}(\frac{Q_i^{(n)}}{\ve})}\bar{L}(\psi_n)\psi_n\to 0,\mbox{ as } n\to \infty.
\end{eqnarray*}
Then we can extract from the sequence $\psi_n(\cdot-\frac{Q_i^{(n)}}{\ve})$ a subsequence which will converge weakly in $H^1(\R^n)$ to $\psi_\infty$, such that
\begin{equation}\label{phi1}
\int_{\R^n}|\nabla \psi_\infty|^2+|\psi_\infty|^2-pw^{p-1}\psi_\infty^2=0,
\end{equation}
and
\begin{equation}\label{phi2}
\int_{\R^n} \psi_\infty \phi_0=\int_{\R^n} \psi_\infty \frac{\partial w}{\partial x_i}=0, \mbox{ for }i=1,\cdots,n.
\end{equation}
From (\ref{phi1}) and (\ref{phi2}), we deduce that $\psi_\infty=0$.

Hence
\begin{equation}
\psi_n\rightharpoonup 0 \mbox{ weakly } \mbox{ in } H^1(\R^n).
\end{equation}
So
\begin{equation}
\int_{B_{\frac{\rho_n-1}{2}}(\frac{Q_i^{(n)}}{\ve})} p\tilde{W}^{p-1}\psi_n^2\to 0 \mbox{ as }n\to \infty.
\end{equation}
We have
\begin{equation}
\|\psi_n\|_{H^1(B_{\frac{\rho_n-1}{2}})}\to 0 \mbox{ as }n\to \infty.
\end{equation}
This contradicts the assumption
\begin{equation}
\|\psi_n\|_{H^1}=1.
\end{equation}

So we get that
\begin{equation}\label{psi1}
\int_{\Omega_\ve} [- \bar{L}(\psi)\psi ] \geq c_0 \|\psi\|^2_{H^1(\Omega_\ve)}.
\end{equation}

From (\ref{psi}) and (\ref{psi1}), we get
\begin{eqnarray}
\|\psi\|^2_{H^1 (\Omega_\epsilon) }&\leq& c(\sum_{ij}|d_{ij} | |\int_{\Omega_\epsilon} \bar{L}(Z_{ij}) \psi | +|\int_{\Omega_\epsilon} \bar{S} \psi | ) \\
&\leq& c(\sum_{ij}|d_{ij}|\|\psi\|_{H^1 (\Omega_\epsilon) }+\|\bar{S}\|_{L^2 (\Omega_\epsilon) }\|\psi\|_{H^1 (\Omega_\epsilon)}).
\end{eqnarray}
So
\begin{eqnarray}\label{psi2}
\|\psi\|_{H^1 (\Omega_\epsilon) }
&\leq& c(\sum_{ij}|d_{ij}| +\|\bar{S}\|_{L^2 (\Omega_\epsilon)}).
\end{eqnarray}
From (\ref{c}) (\ref{d}) (\ref{s}) and (\ref{psi2}), we get that
\begin{eqnarray}
\|\varphi_{k+1}\|_{H^1 (\Omega_\epsilon) }&\leq& c(e^{-\frac{\rho}{2}(1+\xi)}+\|\bar{S}\|_{L^2})\\
&\leq& ce^{-\frac{\rho}{2}(1+\xi)}.
\end{eqnarray}

\end{proof}

\section{The Reduced Problem: A Maximization Procedure}
In this section, we study a maximization problem. Fix $\mathbf{Q}\in \Lambda_k$, we define a new functional
\begin{equation}
\mathcal{M}_\ve(\mathbf{Q})=J_\ve(u_{\ve,\mathbf{Q}})=J_\ve[w_{\ve,\mathbf{Q}}+\phi_{\ve,\mathbf{Q}}]: \Lambda_k \rightarrow \R.
\end{equation}
Define
\begin{equation}
C_k^\ve=\max_{\mathbf{Q}\in\Lambda_k}\{\mathcal{M}_\ve(\mathbf{Q})\}.
\end{equation}
Since $\mathcal{M}_\ve(\mathbf{Q})$ is continuous in $\mathbf{Q}$, the maximization problem has a solution. Let $\mathcal{M}_\ve(\bar{\mathbf{Q}})$ be the maximum where $\bar{\mathbf{Q}}=(\bar{Q}_1,\cdots,\bar{Q}_k)\in \bar{\Lambda}_k$, that is
 \begin{equation}
 \mathcal{M}_\ve(\bar{Q}_1,\cdots,\bar{Q}_k)=\max_{\mathbf{Q}\in \Lambda_k}\mathcal{M}_\ve(\mathbf{Q}),
 \end{equation}
 and we denote the solution by $u_{\ve,\bar{Q}_1,\cdots,\bar{Q}_k}$.

A consequence of Lemma \ref{lemma501} is the following:
\begin{proposition}\label{p501}
Suppose that $  k < \frac{\delta}{\epsilon^n}$ where $\delta$ is sufficiently small (but independent of $\epsilon$). Then it holds
\begin{equation}\label{e501}
C_{k+1}^\ve>C_k^\ve+I(w)-\frac{\gamma}{4}e^{-\rho},
\end{equation}
where $I(w)$ is the energy of $w$,
\begin{equation}
I(w)=\frac{1}{2}\int_{\R^n}(|\nabla w|^2+w^2)-\frac{1}{p+1}\int_{\R^n}w^{p+1}.
\end{equation}
and $\gamma >0$ is defined at (\ref{gammadef}).
\end{proposition}

\medskip

\begin{proof}
We prove it by contradiction. Assume that on  the contrary we have
\begin{equation} \label{assumption}
C_{k+1}^\ve\leq C_k^\ve+I(w)-\frac{\gamma}{4}e^{-\rho}.
\end{equation}

First we claim:

Given $(Q_1,\cdots,Q_k)\in\bar{\Lambda}_k$, there exists $Q_{k+1}\in\Omega$, such that
\begin{equation}
\label{B49}
B_{3\rho\ve}(Q_{k+1})\cap \{Q_1,\cdots,Q_k,\partial \Omega\}=\emptyset.
 \end{equation}

In fact, if not,  we have $k\cdot|B_1|\cdot(3\rho)^n\geq \frac{|\Omega|}{2\ve^n}$. So $k\geq \frac{|\Omega|}{2\times3^n\rho^n\ve^n|B_1|}=\frac{C_{\Omega,n}}{\rho^n\ve^n}$. By the assumption, we have $k\leq \frac{\delta}{\rho^n\ve^n}$ where $\delta$ is sufficiently small. This is a contradiction if we choose $\delta$ so small such that $\delta<C_{\Omega,n}$. So the claimed is proved.

Assume that $(\bar{Q}_1,\cdots,\bar{Q}_k)\in \bar{\Lambda}_k$ is such that $\mathcal{M}_{\ve}(\bar{Q}_1,\cdots,\bar{Q}_k)=\max_{\mathbf{Q}\in \Lambda_k}\mathcal{M}_\ve(\mathbf{Q})=C_k^\ve$, and we denote the solution by $u_{\ve,\bar{Q}_1,\cdots,\bar{Q}_k}$. Let $Q_{k+1}$ be a point satisfying (\ref{B49}). (The existence of $ C_k^\ve$ follows from continuity of  $ {\mathcal M}_\ve$.)

Next we consider the solution concentrates at $(\bar{Q}_1,\cdots,\bar{Q}_k,Q_{k+1})$. As in Section 5,  we decompose  the solution  as \begin{equation}
u_{\ve,\bar{Q}_1,\cdots,\bar{Q}_k,Q_{k+1}}=u_{\ve,\bar{Q}_1,\cdots,\bar{Q}_k}+u_{\ve,Q_{k+1}}+\varphi_{k+1}.
\end{equation}


By the definition of $C_k^{\ve}$, it is easy to see that
\begin{equation}\label{ck1}
C_{k+1}^\ve\geq J_\ve(u_{\ve,\bar{Q}_1,\cdots,Q_{k+1}}).
\end{equation}
Define a cut-off function $\tilde{\chi}$  such that $\tilde{\chi}(x)=\tau(dist(x,\partial B_{\frac{3\rho}{2}}(\frac{Q_{k+1}}{\ve})))$, where $\tau$ is a cutoff function, $\tau(t)=0$ if $t\leq \frac{1}{2}$, $\tau(t)=1$ if $t\geq 1$.

Let us define, now, $\mu=\tilde{\chi} u_{\ve,\bar{Q}_1,\cdots,Q_{k+1}}$.  Then we evaluate $J_\ve(\mu)$:
\begin{eqnarray*}
&&J_\ve(\mu)=J_\ve(\tilde{\chi} u_{\ve,\bar{Q}_1,\cdots,Q_{k+1}})\\
&&=\frac{1}{2}\int_{\Omega_\ve}|\tilde{\chi} \nabla u_{\ve,\bar{Q}_1,\cdots,Q_{k+1}}+u_{\ve,\bar{Q}_1,\cdots,Q_{k+1}}\nabla \tilde{\chi}|^2+\tilde{\chi}^2u_{\ve,\bar{Q}_1,\cdots,Q_{k+1}}^2dx\\
&&-\frac{1}{p+1}\int_{\Omega_{\ve}}\tilde{\chi}^{p+1}u_{\ve,\bar{Q}_1,\cdots,Q_{k+1}}^{p+1}dx\\
&&=\frac{1}{2}\int_{\Omega_\ve}(|\nabla u_{\ve,\bar{Q}_1,\cdots,Q_{k+1}}|^2+u_{\ve,\bar{Q}_1,\cdots,Q_{k+1}}^2)dx-\frac{1}{p+1}\int_{\Omega_\ve}
u_{\ve,\bar{Q}_1,\cdots,Q_{k+1}}^{p+1}dx\\
&&+\frac{1}{2}\int_{\Omega_\ve}|\nabla \tilde{\chi}|^2u_{\ve,\bar{Q}_1,\cdots,Q_{k+1}}^2dx
+\frac{1}{4}\int_{\Omega_\ve}\nabla\tilde{\chi}^2\nabla u_{\ve,\bar{Q}_1,\cdots,Q_{k+1}}^2dx\\
&&+\frac{1}{2}\int_{\Omega_\ve}(\tilde{\chi}^2-1)(|\nabla u_{\ve,\bar{Q}_1,\cdots,Q_{k+1}}|^2+u_{\ve,\bar{Q}_1,\cdots,Q_{k+1}}^2)dx\\
&&+\frac{1}{p+1}\int_{\Omega_\ve}(1-\tilde{\chi}^{p+1})u_{\ve,\bar{Q}_1,\cdots,Q_{k+1}}^{p+1}dx\\
&&= J_\ve(u_{\ve,\bar{Q}_1,\cdots,Q_{k+1}})+\frac{1}{2}\int_{\Omega_\ve}(|\nabla \tilde{\chi}|^2-\frac{1}{2}\Delta \tilde{\chi}^2)u_{\ve,\bar{Q}_1,\cdots,Q_{k+1}}^2dx\\
&&+\frac{1}{2}\int_{\Omega_\ve}(\tilde{\chi}^2-1)(|\nabla u_{\ve,\bar{Q}_1,\cdots,Q_{k+1}}|^2+u_{\ve,\bar{Q}_1,\cdots,Q_{k+1}}^2)dx\\
&&+\frac{1}{p+1}\int_{\Omega_\ve}(1-\tilde{\chi}^{p+1})
u_{\ve,\bar{Q}_1,\cdots,Q_{k+1}}^{p+1}dx.
\end{eqnarray*}
By the definition of the cut-off function $\tilde{\chi}$ and taking into account the exponentially decaying away from the spikes of the function $u_{\ve,\bar{Q}_1,\cdots,Q_{k+1}}$, we have
\begin{eqnarray*}
&&|\frac{1}{2}\int_{\Omega_\ve}(|\nabla \tilde{\chi}|^2-\frac{1}{2}\Delta \tilde{\chi}^2)u_{\ve,\bar{Q}_1,\cdots,Q_{k+1}}^2dx
+\frac{1}{p+1}\int_{\Omega_\ve}(1-\tilde{\chi}^{p+1})u_{\ve,\bar{Q}_1,\cdots,Q_{k+1}}^{p+1}dx\\
&&+\frac{1}{2}\int_{\Omega_\ve}(\tilde{\chi}^2-1)(|\nabla u_{\ve,\bar{Q}_1,\cdots,Q_{k+1}}|^2+u_{\ve,\bar{Q}_1,\cdots,Q_{k+1}}^2)dx|\leq Ce^{-(1+\xi)\rho}
\end{eqnarray*}
for some $\xi>0$.
So we get
\begin{equation}\label{mu}
J_\ve(\mu)= J_\ve(u_{\ve,\bar{Q}_1,\cdots,Q_{k+1}})+O(e^{-(1+\xi)\rho})
\end{equation}
for some $\xi>0$.

On the other hand, one can see that
\begin{equation}
\mu=\mu_1+\mu_2,
\end{equation}
with
\begin{equation}
\mu_1=\left\{\begin{array}{c}
\tilde{\chi} u_{\ve,\bar{Q}_1,\cdots,Q_{k+1}}\ \ \ \ \mbox{ if }x\in B_{\frac{3\rho}{2}}(\frac{Q_{k+1}}{\ve}) \mbox{ and } dist(x,\partial B_{\frac{3\rho}{2}}(\frac{Q_{k+1}}{\ve}))\geq \frac{1}{2}\\
0\ \ \ \mbox{otherwise},
\end{array}
\right.
\end{equation}
and
\begin{equation}
\mu_2=\left\{\begin{array}{c}
\tilde{\chi} u_{\ve,\bar{Q}_1,\cdots,Q_{k+1}}\ \ \ \ \mbox{ if }x\in \Omega_\ve \backslash B_{\frac{3\rho}{2}}(\frac{Q_{k+1}}{\ve}) \mbox{ and } dist(x,\partial B_{\frac{3\rho}{2}}(\frac{Q_{k+1}}{\ve}))\geq \frac{1}{2}\\
0\ \ \ \mbox{ otherwise }.
\end{array}
\right.
\end{equation}
From the definition of $\mu_1$ and $\mu_2$, we have
\begin{equation}\label{mu1mu2}
J_\ve(\mu)=J_\ve(\mu_1+\mu_2)=J_\ve(\mu_1)+J_\ve(\mu_2).
\end{equation}
So we need to evaluate $J_\ve(\mu_1) $ and $J_\ve(\mu_2)$ separately.

First let us consider $J_\ve(\mu_1)$:
\begin{eqnarray*}
&&J_{\ve}(\mu_1)=\frac{1}{2}\int_{B_{\frac{3\rho-1}{2}}(\frac{Q_{k+1}}{\ve})}|\nabla \tilde{\chi} u_{\ve,\bar{Q}_1,\cdots,Q_{k+1}}+\nabla u_{\ve,\bar{Q}_1,\cdots,Q_{k+1}}\tilde{\chi}|^2+|\tilde{\chi} u_{\ve,\bar{Q}_1,\cdots,Q_{k+1}}|^2\\
&&-\frac{1}{p+1}\int_{B_{\frac{3\rho-1}{2}}(\frac{Q_{k+1}}{\ve})}(\tilde{\chi} u_{\ve,\bar{Q}_1,\cdots,Q_{k+1}})^{p+1}dx\\
&&=\frac{1}{2}\int_{B_{\frac{3\rho-1}{2}}(\frac{Q_{k+1}}{\ve})}\tilde{\chi}^2(|\nabla u_{\ve,\bar{Q}_1,\cdots,Q_{k+1}}|^2+u_{\ve,\bar{Q}_1,\cdots,Q_{k+1}}^2)\\
&&-\frac{1}{p+1}\int_{B_{\frac{3\rho-1}{2}}(\frac{Q_{k+1}}{\ve})}\tilde{\chi}^{p+1}
u_{\ve,\bar{Q}_1,\cdots,Q_{k+1}}^{p+1}dx\\
&&+\frac{1}{2}\int_{B_{\frac{3\rho-1}{2}}(\frac{Q_{k+1}}{\ve})}(|\nabla \tilde{\chi}|^2-\frac{1}{2}\Delta \tilde{\chi}^2)u_{\ve,\bar{Q}_1,\cdots,Q_{k+1}}^2dx\\
&&=\frac{1}{2}\int_{B_{\frac{3\rho-1}{2}}(\frac{Q_{k+1}}{\ve})}(|\nabla u_{\ve,\bar{Q}_1,\cdots,Q_{k+1}}|^2+u_{\ve,\bar{Q}_1,\cdots,Q_{k+1}}^2)dx\\
&&-\frac{1}{p+1}\int_{B_{\frac{3\rho-1}{2}}(\frac{Q_{k+1}}{\ve})}
u_{\ve,\bar{Q}_1,\cdots,Q_{k+1}}^{p+1}dx+O(e^{-(1+\xi)\rho})\\
\end{eqnarray*}
\begin{eqnarray*}
&&=\frac{1}{2}\int_{B_{\frac{3\rho-1}{2}}(\frac{Q_{k+1}}{\ve})}(|\nabla u_{\ve,Q_{k+1}}|^2+u_{\ve,Q_{k+1}}^2)
-\frac{1}{p+1}\int_{B_{\frac{3\rho-1}{2}}(\frac{Q_{k+1}}{\ve})} u_{\ve,Q_{k+1}}^{p+1}\\
&&+[\frac{1}{2}\int_{B_{\frac{3\rho-1}{2}}(\frac{Q_{k+1}}{\ve})}(|\nabla u_{\ve,\bar{Q}_1,\cdots,Q_{k+1}}|^2+u_{\ve,\bar{Q}_1,\cdots,Q_{k+1}}^2)dx\\
&&-\frac{1}{2}\int_{B_{\frac{3\rho-1}{2}}(\frac{Q_{k+1}}{\ve})}(|\nabla u_{\ve,Q_{k+1}}|^2+u_{\ve,Q_{k+1}}^2)\\
&&-\frac{1}{p+1}\int_{B_{\frac{3\rho-1}{2}}(\frac{Q_{k+1}}{\ve})}u_{\ve,\bar{Q}_1,\cdots,Q_{k+1}}^{p+1}dx
+\frac{1}{p+1}\int_{B_{\frac{3\rho-1}{2}}(\frac{Q_{k+1}}{\ve})} u_{\ve,Q_{k+1}}^{p+1}dx]\\
&&+O(e^{-(1+\xi)\rho}).\\
\end{eqnarray*}
Using (\ref{vark100}) and (\ref{keyvar}), we obtain
\begin{eqnarray}\label{e501n}
&&|\frac{1}{2}\int_{B_{\frac{3\rho-1}{2}}(\frac{Q_{k+1}}{\ve})}(|\nabla u_{\ve,\bar{Q}_1,\cdots,Q_{k+1}}|^2+u_{\ve,\bar{Q}_1,\cdots,Q_{k+1}}^2)
-\frac{1}{2}\int_{B_{\frac{3\rho-1}{2}}(\frac{Q_{k+1}}{\ve})}(|\nabla u_{\ve,Q_{k+1}}|^2+u_{\ve,Q_{k+1}}^2)\nonumber\\
&&-\frac{1}{p+1}\int_{B_{\frac{3\rho-1}{2}}(\frac{Q_{k+1}}{\ve})}u_{\ve,\bar{Q}_1,\cdots,Q_{k+1}}^{p+1}dx
+\frac{1}{p+1}\int_{B_{\frac{3\rho-1}{2}}(\frac{Q_{k+1}}{\ve})} u_{\ve,Q_{k+1}}^{p+1}dx|\nonumber\\
&&=|\int_{B_{\frac{3\rho-1}{2}}(\frac{Q_{k+1}}{\ve})}\nabla u_{\ve,Q_{k+1}}\nabla(u_{\ve,\bar{Q}_1,\cdots,\bar{Q}_k}+\varphi_{k+1})
+u_{\ve,Q_{k+1}}(u_{\ve,\bar{Q}_1,\cdots,\bar{Q}_k}+\varphi_{k+1})\nonumber\\
&&-u_{\ve,Q_{k+1}}^p(u_{\ve,\bar{Q}_1,\cdots,\bar{Q}_k}+\varphi_{k+1})dx|
+O(e^{-(1+\xi)\rho})\nonumber\\
&&=|\int_{\partial B_{\frac{3\rho-1}{2}}(\frac{Q_{k+1}}{\ve})}\frac{\partial u_{\ve,Q_{k+1}}}{\partial \nu}(u_{\ve,\bar{Q}_1,\cdots,\bar{Q}_k}+\varphi_{k+1})\\
&&-\int_{ B_{\frac{3\rho-1}{2}}(\frac{Q_{k+1}}{\ve})}
S_\ve(u_{\ve,Q_{k+1}})(u_{\ve,\bar{Q}_1,\cdots,\bar{Q}_k}+\varphi_{k+1})|+O(e^{-(1+\xi)\rho})\nonumber\\
&&\leq C\|S_\ve(u_{\ve,Q_{k+1}})\|_{L^2(B_{\frac{3\rho-1}{2}}(\frac{Q_{k+1}}{\ve}))}
(\|(u_{\ve,\bar{Q}_1,\cdots,\bar{Q}_k}\|_{L^2(B_{\frac{3\rho-1}{2}}(\frac{Q_{k+1}}{\ve}))}
+\|\varphi_{k+1})\|_{L^2})\nonumber\\
&&+O(e^{-(1+\xi)\rho}).\nonumber\\
\end{eqnarray}
By (\ref{cijl}), Proposition \ref{p401} and Lemma \ref{lemma501}, we infer that
\begin{eqnarray}
&&\|S_\ve(u_{\ve,Q_{k+1}})\|_{L^2(B_{\frac{3\rho-1}{2}}(\frac{Q_{k+1}}{\ve}))}
(\|(u_{\ve,\bar{Q}_1,\cdots,\bar{Q}_k}\|_{L^2(B_{\frac{3\rho-1}{2}}(\frac{Q_{k+1}}{\ve}))}
+\|\varphi_{k+1})\|_{L^2})\nonumber\\
&&\leq Ce^{-(1+\xi)\rho}.
\end{eqnarray}
Again by Lemma \ref{lemma2} and Proposition \ref{p401}, we have
\begin{eqnarray*}
&&|\frac{1}{2}\int_{\Omega_\ve\backslash B_{\frac{3\rho-1}{2}}(\frac{Q_{k+1}}{\ve})}(|\nabla u_{\ve,Q_{k+1}}|^2+u_{\ve,Q_{k+1}}^2)
-\frac{1}{p+1}\int_{\Omega_\ve\backslash B_{\frac{3\rho-1}{2}}(\frac{Q_{k+1}}{\ve})} u_{\ve,Q_{k+1}}^{p+1}|\\
&&\leq ce^{-(1+\xi)\rho}.
\end{eqnarray*}

Combining the above, we obtain
\begin{eqnarray}\label{mu1}
J_\ve(\mu_1)=J_\ve(u_{\ve,Q_{k+1}})+O(e^{-(1+\xi)\rho}).
\end{eqnarray}
Similar to (\ref{e501n}), we have
\begin{eqnarray}
J_\ve(u_{\ve,Q_{k+1}})&=&J_\ve(w_{\ve,Q_{k+1}}+\phi_{\ve,Q_{k+1}})\\
&=&J_\ve(w_{\ve,Q_{k+1}})+O(e^{-(1+\xi)\rho}).\nonumber
\end{eqnarray}
By the definition of $w_{\ve,Q_{k+1}}$, we get
\begin{eqnarray*}
&&J_\ve(w_{\ve,Q_{k+1}})\\
&&=\frac{1}{2}\int_{\Omega_\ve}w_{Q_{k+1}}^pw_{\ve,Q_{k+1}}dx
-\frac{1}{p+1}\int_{\Omega_\ve}w_{\ve,Q_{k+1}}^{p+1}dx\\
&&=\frac{1}{2}\int_{\Omega_\ve}w_{Q_{k+1}}^{p+1}dx
-\frac{1}{p+1}\int_{\Omega_\ve}w_{Q_{k+1}}^{p+1}dx\\
&&-\frac{1}{2}\int_{\Omega_\ve}w_{Q_{k+1}}^p\varphi_{\ve,Q_{k+1}}dx
-\frac{1}{p+1}\int_{\Omega_\ve}w_{\ve,Q_{k+1}}^{p+1}-w_{Q_{k+1}}^{p+1}dx.\\
\end{eqnarray*}
Note that
\begin{eqnarray*}
&&\int_{\Omega_\ve}(\frac{1}{2}-\frac{1}{p+1})w_{Q_{k+1}}^{p+1}dx\\
&&=\int_{\R^n}(\frac{1}{2}-\frac{1}{p+1})w_{Q_{k+1}}^{p+1}dx
-\int_{\R^n\backslash \Omega_\ve}(\frac{1}{2}-\frac{1}{p+1})w_{Q_{k+1}}^{p+1}dx\\
&&=I(w)+O(e^{-(1+\xi)\rho})
\end{eqnarray*}
and
\begin{eqnarray*}
|\int_{\Omega_\ve}\frac{1}{p+1}w_{\ve,Q_{k+1}}^{p+1}-\frac{1}{p+1}w_{Q_{k+1}}^{p+1}
+w_{Q_{k+1}}^p\varphi_{\ve,Q_{k+1}}dx|&&\leq C\int_{\Omega_\ve}w_{Q_{k+1}}^{p-1}\varphi_{\ve,Q_{k+1}}^2dx\\
&&\leq Ce^{-(1+\xi)\rho}.
\end{eqnarray*}

So by Lemma \ref{lemma3}, we get
\begin{eqnarray}\label{wve}
J_\ve(w_{\ve,Q_{k+1}})&&=I(w)-\frac{1}{2}B_\ve(Q_{k+1})+O(e^{-(1+\xi)\rho})\\
&&=I(w)+O(e^{-(1+\xi)\rho}).\nonumber
\end{eqnarray}

(\ref{mu1}) and (\ref{wve}) yield

\begin{eqnarray}\label{qk1}
J_\ve(u_{\ve,Q_{k+1}})&=&I(w)-\frac{1}{2}B_\ve(Q_{k+1})+O(e^{-(1+\xi)\rho})\\
&=&I(w)+O(e^{-(1+\xi)\rho}).\nonumber
\end{eqnarray}
Now let us consider $J_\ve(\mu_2)$:
\begin{eqnarray*}
&&J_\ve(\mu_2)\\
&&=\frac{1}{2}\int_{\Omega_\ve\backslash B_{\frac{3\rho+1}{2}}(\frac{Q_{k+1}}{\ve})}|\nabla \tilde{\chi} u_{\ve,\bar{Q}_1,\cdots,Q_{k+1}}+\nabla u_{\ve,\bar{Q}_1,\cdots,Q_{k+1}}\tilde{\chi}|^2+|\tilde{\chi} u_{\ve,\bar{Q}_1,\cdots,Q_{k+1}}|^2dx\\
&&-\frac{1}{p+1}\int_{\Omega_\ve\backslash B_{\frac{3\rho+1}{2}}(\frac{Q_{k+1}}{\ve})}(\tilde{\chi} u_{\ve,\bar{Q}_1,\cdots,Q_{k+1}})^{p+1}dx\\
&&=\frac{1}{2}\int_{\Omega_\ve\backslash B_{\frac{3\rho+1}{2}}(\frac{Q_{k+1}}{\ve})}(|\nabla u_{\ve,\bar{Q}_1,\cdots,Q_{k+1}}|^2+u_{\ve,\bar{Q}_1,\cdots,Q_{k+1}}^2)dx\\
&&-\frac{1}{p+1}\int_{\Omega_\ve\backslash B_{\frac{3\rho+1}{2}}(\frac{Q_{k+1}}{\ve})}u_{\ve,\bar{Q}_1,\cdots,Q_{k+1}}^{p+1}dx+O(e^{-(1+\xi)\rho})\\
&&=\frac{1}{2}\int_{\Omega_\ve\backslash B_{\frac{3\rho+1}{2}}(\frac{Q_{k+1}}{\ve})}|\nabla u_{\ve,\bar{Q}_,\cdots,\bar{Q}_k}|^2
+u_{\ve,\bar{Q}_1,\cdots,\bar{Q}_k}^2dx\\
&&-\frac{1}{p+1}\int_{\Omega_\ve\backslash B_{\frac{3\rho+1}{2}}(\frac{Q_{k+1}}{\ve})}u_{\ve,\bar{Q}_1,\cdots,\bar{Q}_k}^{p+1}dx\\
&&+[\frac{1}{2}\int_{\Omega_\ve\backslash B_{\frac{3\rho+1}{2}}(\frac{Q_{k+1}}{\ve})}(|\nabla u_{\ve,\bar{Q}_1,\cdots,Q_{k+1}}|^2+u_{\ve,\bar{Q}_1,\cdots,Q_{k+1}}^2)dx\\
&&-\frac{1}{2}\int_{\Omega_\ve\backslash B_{\frac{3\rho+1}{2}}(\frac{Q_{k+1}}{\ve})}|\nabla u_{\ve,\bar{Q}_,\cdots,\bar{Q}_k}|^2
+u_{\ve,\bar{Q}_1,\cdots,\bar{Q}_k}^2dx\\
&&-\frac{1}{p+1}\int_{\Omega_\ve\backslash B_{\frac{3\rho+1}{2}}(\frac{Q_{k+1}}{\ve})}u_{\ve,\bar{Q}_1,\cdots,Q_{k+1}}^{p+1}dx
+\frac{1}{p+1}\int_{\Omega_\ve\backslash B_{\frac{3\rho+1}{2}}(\frac{Q_{k+1}}{\ve})}u_{\ve,\bar{Q}_1,\cdots,\bar{Q}_k}^{p+1}dx]\\
&&+O(e^{-(1+\xi)\rho}).\\
\end{eqnarray*}

Similar to (\ref{e501n}), we can get
\begin{eqnarray*}
&&|\frac{1}{2}\int_{\Omega_\ve\backslash B_{\frac{3\rho+1}{2}}(\frac{Q_{k+1}}{\ve})}(|\nabla u_{\ve,\bar{Q}_1,\cdots,Q_{k+1}}|^2+u_{\ve,\bar{Q}_1,\cdots,Q_{k+1}}^2)dx\\
&&-\frac{1}{2}\int_{\Omega_\ve\backslash B_{\frac{3\rho+1}{2}}(\frac{Q_{k+1}}{\ve})}|\nabla u_{\ve,\bar{Q}_,\cdots,\bar{Q}_k}|^2+u_{\ve,\bar{Q}_1,\cdots,\bar{Q}_k}^2dx\\
&&-\frac{1}{p+1}\int_{\Omega_\ve\backslash B_{\frac{3\rho+1}{2}}(\frac{Q_{k+1}}{\ve})}u_{\ve,\bar{Q}_1,\cdots,Q_{k+1}}^{p+1}dx+\frac{1}{p+1}\int_{\Omega_\ve\backslash B_{\frac{3\rho+1}{2}}(\frac{Q_{k+1}}{\ve})}u_{\ve,\bar{Q}_1,\cdots,\bar{Q}_k}^{p+1}dx|\\
&&=|\int_{\Omega_\ve\backslash B_{\frac{3\rho+1}{2}}(\frac{Q_{k+1}}{\ve})}S_\ve(u_{\ve,\bar{Q}_1\cdots,\bar{Q}_k})
(u_{\ve,Q_{k+1}}+\varphi_{k+1})dx|+e^{-(1+\xi)\rho}\\
&&=|\sum_{i=1,\cdots,k, \  j=1,\cdots,n}c_{ij}
\int_{\Omega_\ve\backslash B_{\frac{3\rho+1}{2}}(\frac{Q_{k+1}}{\ve})}Z_{ij}(u_{\ve,Q_{k+1}}+\varphi_{k+1})dx|.
\end{eqnarray*}
By Lemma \ref{lemma501}, (\ref{decom}), (\ref{c}), (\ref{d}) and (\ref{cijl}), we have
\begin{eqnarray*}
&&|\sum_{i=1,\cdots,k, \  j=1,\cdots,n}c_{ij}\int_{\Omega_\ve\backslash B_{\frac{3\rho+1}{2}}(\frac{Q_{k+1}}{\ve})}Z_{ij}\varphi_{k+1}dx|\\
&&=|\sum_{i=1,\cdots,k, \ j=1,\cdots,n}c_{ij}\int_{\Omega_\ve\backslash B_{\frac{3\rho+1}{2}}(\frac{Q_{k+1}}{\ve})}Z_{ij}(\sum c_i\phi_i+\sum_{ij}d_{ij}Z_{ij})dx|\\
&&\leq c \sup_{ij}|c_{ij}|\sum (|c_i|+|d_{ij}|)\\
&&\leq ce^{-(1+\xi)\rho},
\end{eqnarray*}
\begin{eqnarray*}
&&|\sum_{i=1,\cdots,k, \ j=1,\cdots,n}c_{ij}\int_{\Omega_\ve\backslash B_{\frac{3\rho+1}{2}}(\frac{Q_{k+1}}{\ve})}Z_{ij}u_{\ve,Q_{k+1}}dx|\leq ce^{-(1+\xi)\rho},
\end{eqnarray*}
and
\begin{equation*}
|\int_{B_{\frac{3\rho+1}{2}}(\frac{Q_{k+1}}{\ve})}|\nabla u_{\ve,\bar{Q}_1,\cdots,\bar{Q}_k}|^2+u_{\ve,\bar{Q}_1,\cdots,\bar{Q}_k}^2-\frac{1}{p+1}u_{\ve,\bar{Q}_1,\cdots,\bar{Q}_k}^{p+1}dx|\leq Ce^{-(1+\xi)\rho}.
\end{equation*}
Recalling that
\begin{equation}
C^k_\ve=J_\ve(u_{\ve,\bar{Q}_1,\cdots,\bar{Q}_k}),
\end{equation}
we get
\begin{equation}\label{mu2}
J_\ve(\mu_2)= C_k^\ve+O(e^{-(1+\xi)\rho}).
\end{equation}
Thus combining (\ref{ck1}), (\ref{mu}), (\ref{mu1mu2}), (\ref{mu1}), (\ref{qk1}) and (\ref{mu2}), we have
\begin{eqnarray*}
J_\ve(\mu)&=&J_\ve(\mu_1+\mu_2)\\
&=&J_\ve(\mu_1)+J_\ve(\mu_2)\\
&=& C_k^\ve+I(w)+O(e^{-(1+\xi)\rho})\\
&=&J_\ve(u_{\ve,\bar{Q}_1,\cdots,Q_{k+1}})+O(e^{-(1+\xi)\rho})\\
&\leq& C_{k+1}^\ve+O(e^{-(1+\xi)\rho}).
\end{eqnarray*}
Thus,
\begin{eqnarray*}
C_{k+1}^\ve\geq C_k^\ve+I(w)+O(e^{-(1+\xi)\rho}),
\end{eqnarray*}
a contradiction with the assumption (\ref{assumption}).
\end{proof}
\begin{remark}\label{remark501}
From the proof above, we may take $\delta(n,p,\Omega)=\frac{\delta_0}{\rho_0^n}<<\frac{|\Omega|}{2\times 3^n|B_1|\rho_0^n}$ for some $\delta_0>0$ small, where $\rho_0$ is as in Section 4.

\end{remark}
Next we have the following Proposition:
\begin{proposition}\label{p502}
The maximization problem
\begin{equation}
\max_{\mathbf{Q}\in \bar{\Lambda}_k} \mathcal{M}_\ve(\mathbf{Q})
\end{equation}
has a solution $\mathbf{Q}^\ve \in \Lambda_k^\circ$, i.e., the interior of $\Lambda_k$.
\end{proposition}
\begin{proof}
We prove it by contradiction again.
If $\mathbf{Q}^\ve=(\bar{Q}_1,\cdots,\bar{Q}_k) \in \partial \Lambda_k$, then either there exists $(i,j)$ such that $|Q_i-Q_j|=\ve \rho$ or $|Q_i-Q_j^*|=\ve \rho$. Without loss of generality, we assume $(i,j)=(i,k)$. We have
\begin{eqnarray*}
J_\ve(u_{\ve,\bar{Q}_1,\cdots,\bar{Q}_k})
&=&J_\ve(u_{\ve,\bar{Q}_1,\cdots,\bar{Q}_{k-1}}+u_{\ve,\bar{Q}_k}
+\varphi_k)\\
&=&J_\ve(u_{\ve,\bar{Q}_1,\cdots,\bar{Q}_{k-1}}+u_{\ve,\bar{Q}_k})
+\int_{\Omega_\ve}\nabla(u_{\ve,\bar{Q}_1,\cdots,\bar{Q}_{k-1}}+u_{\ve,\bar{Q}_k})\nabla \varphi_k\\
&+&(u_{\ve,\bar{Q}_1,\cdots,\bar{Q}_{k-1}}+u_{\ve,\bar{Q}_k})\varphi_k
-(u_{\ve,\bar{Q}_1,\cdots,\bar{Q}_{k-1}}+u_{\ve,\bar{Q}_k})^p\varphi_kdx\\
&+&O(\|\varphi_k\|^2_{H^1})\\
&=&J_\ve(u_{\ve,\bar{Q}_1,\cdots,\bar{Q}_{k-1}}+u_{\ve,\bar{Q}_k})\\
&-&\int_{\Omega_\ve}S_\ve(u_{\ve,\bar{Q}_1,\cdots,\bar{Q}_{k-1}}
+u_{\ve,\bar{Q}_k})\varphi_kdx+O(\|\varphi_k\|^2_{H^1}).\\
\end{eqnarray*}
Observe that
\begin{eqnarray*}
&&S_\ve(u_{\ve,\bar{Q}_1,\cdots,\bar{Q}_{k-1}}
+u_{\ve,\bar{Q}_k})\\
&&=(u_{\ve,\bar{Q}_1,\cdots,\bar{Q}_{k-1}}
+u_{\ve,\bar{Q}_k})^p-u_{\ve,\bar{Q}_1,\cdots,\bar{Q}_{k-1}}^p
-u_{\ve,\bar{Q}_k}^p+\sum_{i=1,\cdots,k,j=1,\cdots,n}c_{ij}Z_{ij},
\end{eqnarray*}
for some $\{c_{ij}\}$ which  satisfies
\begin{equation}\label{cij}
|c_{ij}|\leq ce^{-(1+\xi)\frac{\rho}{2}}.
\end{equation}
Using Lemma \ref{lemma501}, we obtain
\begin{eqnarray*}
&&|\sum_{ij}c_{ij}\int_{\Omega_\ve}Z_{ij}\varphi_k dx|\\
&&=|\sum_{ij}c_{ij}\int_{\Omega_\ve}Z_{ij}(\psi+\sum c_i\phi_i+\sum_{ij}d_{ij}Z_{ij})dx|\\
&&=|\sum_{ij}c_{ij}\int_{\Omega_\ve}Z_{ij}(c_i\phi_i+d_{ij}Z_{ij})dx|\\
&&\leq \sup_{ij}|c_{ij}|\sum_{ij}(|c_i|+|d_{ij}|)\\
&&\leq ce^{-(1+\xi)\rho}
\end{eqnarray*}
by (\ref{c}), (\ref{d}) and (\ref{cij}).

Using the estimate (\ref{s}) and Lemma \ref{lemma501}, we find
\begin{eqnarray*}
&&|\int_{\Omega_\ve}((u_{\ve,\bar{Q}_1,\cdots,\bar{Q}_{k-1}}
+u_{\ve,\bar{Q}_k})^p-u_{\ve,\bar{Q}_1,\cdots,\bar{Q}_{k-1}}^p
-u_{\ve,\bar{Q}_k}^p)\varphi_kdx|\\
&&\leq c\|\bar{S}\|_{L^2}\|\varphi_k\|_{H^1}\leq ce^{-(1+\xi)\rho}.
\end{eqnarray*}
This implies that
\begin{eqnarray*}
&&J_\ve(u_{\ve,\bar{Q}_1,\cdots,\bar{Q}_k})\\
&&=J_\ve(u_{\ve,\bar{Q}_1,\cdots,\bar{Q}_{k-1}}+u_{\ve,\bar{Q}_k})\\
&&-\int_{\Omega_\ve}S_\ve(u_{\ve,\bar{Q}_1,\cdots,\bar{Q}_{k-1}}
+u_{\ve,\bar{Q}_k})\varphi_kdx+O(\|\varphi_k\|^2_{H^1})\\
&&=J_\ve(u_{\ve,\bar{Q}_1,\cdots,\bar{Q}_{k-1}}+u_{\ve,\bar{Q}_k})+O(e^{-(1+\xi)\rho}).
\end{eqnarray*}

Next we estimate
\begin{eqnarray*}
&&J_\ve(u_{\ve,\bar{Q}_1,\cdots,\bar{Q}_{k-1}}+u_{\ve,\bar{Q}_k})\\
&&=J_\ve(u_{\ve,\bar{Q}_1,\cdots,\bar{Q}_{k-1}})+
J_\ve(u_{\ve,\bar{Q}_k})\\
&&+\int_{\Omega_\ve}\nabla u_{\ve,\bar{Q}_1,\cdots,\bar{Q}_{k-1}}\nabla u_{\ve,\bar{Q}_k}
+u_{\ve,\bar{Q}_1,\cdots,\bar{Q}_{k-1}}u_{\ve,\bar{Q}_k}dx\\
&&-\frac{1}{p+1}\int_{\Omega_\ve}(u_{\ve,\bar{Q}_1,\cdots,\bar{Q}_{k-1}}+u_{\ve,\bar{Q}_k})^{p+1}
-u_{\ve,\bar{Q}_1,\cdots,\bar{Q}_{k-1}}^{p+1}-u_{\ve,\bar{Q}_k}^{p+1}dx,\\
\end{eqnarray*}
and
\begin{eqnarray*}
&&\int_{\Omega_\ve}\nabla u_{\ve,\bar{Q}_1,\cdots,\bar{Q}_{k-1}}\nabla u_{\ve,\bar{Q}_k}
+u_{\ve,\bar{Q}_1,\cdots,\bar{Q}_{k-1}}u_{\ve,\bar{Q}_k}dx\\
&&-\frac{1}{p+1}\int_{\Omega_\ve}(u_{\ve,\bar{Q}_1,\cdots,\bar{Q}_{k-1}}+u_{\ve,\bar{Q}_k})^{p+1}
-u_{\ve,\bar{Q}_1,\cdots,\bar{Q}_{k-1}}^{p+1}-u_{\ve,\bar{Q}_k}^{p+1}dx\\
&&=\int_{\Omega_\ve}(u_{\ve,\bar{Q}_{k+1}}^p-\sum_{l=1}^n c_{kl}Z_{kl})u_{\ve,\bar{Q}_1,\cdots,\bar{Q}_{k-1}}dx\\
&&-\int_{\Omega_\ve}u_{\ve,\bar{Q}_1,\cdots,\bar{Q}_{k-1}}^pu_{\ve,\bar{Q}_k}
+u_{\ve,\bar{Q}_k}^pu_{\ve,\bar{Q}_1,\cdots,\bar{Q}_{k-1}}dx+O(e^{-(1+\xi)\rho})\\
&&=-\int_{\Omega_\ve}u_{\ve,\bar{Q}_1,\cdots,\bar{Q}_{k-1}}^pu_{\ve,\bar{Q}_k}dx
+O(e^{-(1+\xi)\rho})
\end{eqnarray*}
by (\ref{cijl}) in Section 3.

The above three identities imply that
\begin{eqnarray}\label{qk2}
J_\ve(u_{\ve,\bar{Q}_1,\cdots,\bar{Q}_k})
&&=J_\ve(u_{\ve,\bar{Q}_1,\cdots,\bar{Q}_{k-1}})+
J_\ve(u_{\ve,\bar{Q}_k})\\
&&-\int_{\Omega_\ve}u_{\ve,\bar{Q}_1,\cdots,\bar{Q}_{k-1}}^pu_{\ve,\bar{Q}_k}dx
+O(e^{-(1+\xi)\rho}).\nonumber
\end{eqnarray}
Since
\begin{eqnarray}\label{qiqj}
&&\int_{\Omega_\ve}u_{\ve,\bar{Q}_1,\cdots,\bar{Q}_{k-1}}^pu_{\ve,\bar{Q}_k}dx\\
&&=\int_{\Omega_\ve}(\sum_{i=1}^{k-1}w_{\ve,\bar{Q}_i}+\phi_{\ve,\bar{Q}_1,\cdots,\bar{Q}_{k-1}})^p(
w_{\ve,\bar{Q}_k}+\phi_{\ve,\bar{Q}_k})dx\nonumber\\
&&\geq\int_{\Omega_\ve}w_{\ve,\bar{Q}_i}^pw_{\ve,\bar{Q}_k}dx+O(e^{-(1+\xi)\rho}),\nonumber
\end{eqnarray}
using (\ref{qk1}) (\ref{qk2}) and (\ref{qiqj}) , one can get

\begin{eqnarray*}
J_\ve(u_{\ve,\bar{Q}_1,\cdots,\bar{Q}_k})
&&\leq C_{k-1}^\ve+I(w)-\frac{1}{2}B_\ve(\bar{Q}_k)
-\int_{\Omega_\ve}w_{\ve,\bar{Q}_i}^pw_{\ve,\bar{Q}_k}dx+O(e^{-(1+\xi)\rho}).
\end{eqnarray*}
If either there exists $(i,k)$ such that $|Q_i-Q_k|=\ve \rho$ or $|Q_k-Q_k^*|=\ve \rho$, by Lemma \ref{lemma3}, we can get that
\begin{equation}
J_\ve(u_{\ve,\bar{Q}_1,\cdots,\bar{Q}_k})\leq C_{k-1}^\ve+I(w)-(\frac{\gamma}{2}+O(\frac{1}{\sqrt{\rho}}))e^{-\rho}
+O(e^{-(1+\xi)\rho}).
\end{equation}
Thus
\begin{eqnarray*}
C_k^\ve=\mathcal{M}_\ve(\mathbf{Q}^\ve)\leq C_{k-1}^\ve+I(w)-\frac{\gamma}{4}e^{-\rho}.
\end{eqnarray*}
We reach a
contradiction with Proposition \ref{p501}.

\end{proof}
\section{Proof of Theorem \ref{theo1}}
In this section, we apply the results in Section 4, Section 5 and Section 6 to prove Theorem \ref{theo1}. The proof is similar to \cite{LNW}.

\noindent
{\bf Proof of Theorem \ref{theo1}:} By Proposition \ref{p401} in Section 4, there exists $\ve_0,\rho_0$ such that for $0<\ve<\ve_0,\rho>\rho_0$, we have $C^1$ map which, to any $\mathbf{Q} \in \Lambda_k$, associates $\phi_{\ve,\mathbf{Q}}$ such that
\begin{equation}
S_\ve(w_{\ve,\mathbf{Q}}+\phi_{\ve,\mathbf{Q}})=\sum_{i=1,\cdots,k,j=1,\cdots,n}c_{ij}Z_{ij},\ \
\int_{\Omega_\ve}\phi_{\ve,\mathbf{Q}}Z_{ij}dx=0,
\end{equation}
for some constants $\{c_{ij}\}\in \R^{kn}$.

From Proposition \ref{p502} in Section 6, there is a $\mathbf{Q}^\ve \in \Lambda_k^\circ$ that achieves the maximum for the maximization problem in Proposition \ref{p502}. Let $u_\ve=w_{\ve,\mathbf{Q}^\ve}+\phi_{\ve,\mathbf{Q}^\ve}$. Then we have
\begin{equation}
D_{Q_{ij}}|_{Q_i=Q_i^\ve}\mathcal{M}_\ve(\mathbf{Q}^\ve)=0,\  \ i=1,\cdots,k, \ \ j=1,\cdots,n.
\end{equation}
Hence we have
\begin{eqnarray*}
&&\int_{\Omega_\ve}\nabla u_\ve \nabla \frac{\partial (w_{\ve,\mathbf{Q}}+\phi_{\ve,\mathbf{Q}})}{\partial Q_{ij}}|_{Q_i=Q_i^\ve}+u_\ve\frac{\partial (w_{\ve,\mathbf{Q}}+\phi_{\ve,\mathbf{Q}})}{\partial Q_{ij}}|_{Q_i=Q_i^\ve}\\
&&-u_\ve^p\frac{\partial (w_{\ve,\mathbf{Q}}+\phi_{\ve,\mathbf{Q}})}{\partial Q_{ij}}|_{Q_i=Q_i^\ve}=0,
\end{eqnarray*}
which gives
\begin{equation}\label{e601}
\sum_{i=1,\cdots,k, \ j=1,\cdots,n}c_{ij}\int_{\Omega_\ve} Z_{ij}\frac{\partial (w_{\ve,\mathbf{Q}}+\phi_{\ve,\mathbf{Q}})}{\partial Q_{sl}}|_{Q_s=Q_s^\ve}=0,
\end{equation}
for $s=1,\cdots,k, l=1,\cdots,n$.
We claim that (\ref{e601}) is a diagonally dominant system.  In fact, since $\int_{\Omega_\ve} \phi_{\ve,\mathbf{Q}}Z_{sl}dx=0$, we have that
\begin{equation*}
\int_{\Omega_\ve}Z_{sl}\frac{\partial \phi_{\ve,\mathbf{Q}}}{\partial Q_{ij}}|_{Q_i=Q_i^\ve}
=-\int_{\Omega_\ve}\phi_{\ve,\mathbf{Q}}\frac{\partial Z_{sl}}{\partial Q_{ij}}=0, \mbox{ if }s\neq i.
\end{equation*}
If $s=i$, we have
\begin{eqnarray*}
|\int_{\Omega_\ve}Z_{il}\frac{\partial \phi_{\ve,\mathbf{Q}}}{\partial Q_{ij}}|_{Q_i=Q_i^\ve}|
=|-\int_{\Omega_\ve}\phi_{\ve,\mathbf{Q}}\frac{\partial Z_{il}}{\partial Q_{ij}}|\\
\leq C\ve^{-1}\|\phi_{\ve,\mathbf{Q}}\|_*=O(\ve^{-1}e^{-\frac{\rho}{2}(1+\xi)}).
\end{eqnarray*}
For $s\neq i$, we have
\begin{equation*}
\int_{\Omega_\ve}Z_{sl}\frac{\partial w_{\ve,\mathbf{Q}}}{\partial Q_{ij}}=O(\ve^{-1}e^{-\frac{\eta|Q_i-Q_s|}{\ve}}).
\end{equation*}
For $s=i$, recall the definition of $Z_{ij}$, we have
\begin{equation}\label{diagonal}
\int_{\Omega_\ve}Z_{sl}\frac{\partial w_{\ve,\mathbf{Q}}}{\partial Q_{sj}}=
-\ve^{-1}\delta_{lj}\int_{\R^n}(\frac{\partial w}{\partial y_j})^2+O(\ve^{-1}e^{-\rho}).
\end{equation}
For each $(s,l)$, the off-diagonal term gives
\begin{eqnarray}\label{offdiagonal}
&&\sum_{s\neq i}\int_{\Omega_\ve}Z_{sl}\frac{\partial (w_{\ve,\mathbf{Q}}+\phi_{\ve,\mathbf{Q}}) }{\partial Q_{ij}}|_{Q_i=Q_i^\ve}+\sum_{s=i,l\neq j}\int_{\Omega_\ve}Z_{sl}\frac{\partial (w_{\ve,\mathbf{Q}}+\phi_{\ve,\mathbf{Q}})}{\partial Q_{sj} } |_{Q_i=Q_i^\ve}\nonumber\\
&&=\ve^{-1}(O(e^{-\eta\rho})+O(e^{-\frac{\rho}{2}})+O(e^{-\rho}))\\
&&=\ve^{-1}O(e^{-\eta\rho}),\nonumber
\end{eqnarray}
for some $\eta>0$.

So from (\ref{diagonal}) and (\ref{offdiagonal}), we can see that equation (\ref{e601}) becomes a system of homogeneous equations for $c_{sl}$, and the matrix of the system is nonsingular. So $c_{sl}=0$ for $s=1,\cdots,k, l=1,\cdots,n$. Hence $u_\ve=w_{\ve,\mathbf{Q}^\ve}+\phi_{\ve,\mathbf{Q}^\ve}$ is a solution of (\ref{pp}).

Similar to the argument in Section 6 of \cite{LNW}, one can get that $u_\ve>0$ and it has exactly $k$ local maximum points for $\ve$ small and $\rho $ large enough.

\end{document}